\documentclass[a4paper,11pt]{article}
\usepackage[utf8]{inputenc}
\usepackage{amssymb,amsmath,mathrsfs,amsthm,bbm}
\usepackage{indentfirst}
\usepackage{geometry}

\usepackage{hyperref} 
\hypersetup{colorlinks=true,allcolors=blue}
\usepackage{hypcap}

\newtheorem{thm}{Theorem}[section]
\newtheorem{defi}[thm]{Definition}
\newtheorem{prop}[thm]{Proposition}
\newtheorem{cor}[thm]{Corollary}
\newtheorem{exa}[thm]{Example}
\newtheorem{lem}[thm]{Lemma}
\newtheorem{rem}[thm]{Remark}

\makeatletter
\newtheorem*{rep@theorem}{\rep@title}
\newcommand{\newreptheorem}[2]{%
	\newenvironment{rep#1}[1]{%
		\def\rep@title{#2 \ref{##1}}%
		\begin{rep@theorem}}%
		{\end{rep@theorem}}}
	\makeatother

\newreptheorem{prop}{Proposition*}

\numberwithin{equation}{section}

\renewcommand{\rm}[1]{\mathrm{#1}}
\renewcommand{\cal}[1]{\mathcal{#1}}

\newcommand{\bb}[1]{\mathbb{#1}}

\newcommand{\bm}[1]{\mathbbm{#1}}

\newcommand{\dd}{\mathrm{d}}

\newcommand{\supp}{\rm{supp}}

\def\slr{\rm{SL}_2(\bb R)}
\def\pso{\rm{PSO}_2}
\def\so{\rm{SO}_2}
\def\projr{\bb P^1}
\def\linf{L^\infty_\xi\cal H^\gamma }
\def\Res{E}
\def\Cut{E_C}
\def\cartwo{\varLambda}
\def\sgn{\rm{sign}}
\def\mtxy{M_t^+(x,y)}
\def\ntint{I_t}
\def\glg{j}
\def\oexp{O_{\exp}}
\def\oexpe{O_{\exp,\epsilon}}
\def\lf{{Lip}}
\def\nte{[\frac{t}{\sigma_\mu+\epsilon}-1,\frac{t}{\sigma_\mu-\epsilon}]}
\def\ck{{Lip}}
\def\Inv{E_I}
\def\Car{E_P}
\def\Lip{{L^\infty Lip}}
\def\pexp{B_s}
\def\pexpo{B}

\begin{document}
\bibliographystyle{alpha}
\title{\textbf{Decrease of Fourier coefficients of stationary measures}}
\author{Jialun LI}
\date{}
\maketitle


\begin{abstract}
Let $\mu$ be a Borel probability measure on $\slr$ with a finite exponential moment, and assume that the subgroup $\Gamma_{\mu}$ generated by the support of $\mu$ is Zariski dense. Let $\nu$ be the unique $\mu-$stationary measure on $\bb P^1$. We prove that the Fourier coefficients $\hat{\nu}(k)$ of $\nu$ converge to $0$ as $|k|$ tends to infinity. Our proof relies on a generalized renewal theorem for the Cartan projection.
\end{abstract}

\section{Introduction}

Let $\mu$ be a Borel probability measure on $\slr$. The linear action of $\slr$ on $\bb R^2$ induces an action on $\bb P^1=\bb P(\bb R^2)$.
For a Borel probability measure $\nu$ on $\projr$, we define its convolution with $\mu$ by
\begin{equation*}
\mu*\nu=\int_{\slr} g_*\nu\dd\mu(g),
\end{equation*}
where $g_*\nu$ is the pushforward of $\nu$ by $g$.
The measure $\nu$ is called $\mu-$stationary if $\mu*\nu=\nu$. We add the condition that the subgroup $\Gamma_{\mu}$ generated by the support of $\mu$ is Zariski dense in $\slr$. In the case of $\slr$, Zariski density is equivalent to unsolvability. When $\Gamma_{\mu}$ is Zariski dense in $\slr$, there is a unique $\mu-$stationary measure (see \cite{furstenberg1963noncommuting},\cite{guivarc1985frontiere}).

This stationary measure is also called the Furstenberg measure. It was first considered by Furstenberg in the study of the noncommutative law of large numbers. The stationary measure takes part in the subtle properties of random products of matrices. Please see \cite{furstenberg1963noncommuting},\cite{guivarc1985frontiere} and \cite{bougerol1985products}.

In this paper, we are interested in the decay of the Fourier coefficients of stationary measures. The action of $\pso=\so/\{\pm Id\}$ on $\projr$ is transitive and free. We fix the point $x_o=[1:0]$ in $\projr$, then identify $\projr$ as the orbit space $\pso x_o$. 
As a group, $\pso$ is isomorphic to the circle $\bb T\simeq\bb R/\pi\bb Z$. This is given by the map from $\bb T$ to $\pso$,
\[\theta\mapsto \begin{pmatrix}
\cos\theta & -\sin\theta \\ \sin\theta & \cos\theta
\end{pmatrix}/\{\pm Id\}. \]
So we have a homeomorphism from $\bb T$ to $\bb P^1$, that is $\theta\mapsto[\cos\theta:\sin\theta]$. We can define the Fourier coefficients of the stationary measure $ \nu$ by the following formula
\[\widehat{ \nu}(k)=\int_{\bb T}e^{2ik\theta}\dd \nu(\theta). \]
We also demand that $\mu$ has a finite exponential moment, which means that there exists a constant $\epsilon_1>0$ such that $\int\|g\|^{\epsilon_1}\dd\mu(g)<\infty$. We will prove
\begin{thm}\label{thm:fourier}
Let $\mu$ be a Borel probability measure on $\slr$ with a finite exponential moment, and assume that the subgroup $\Gamma_{\mu}$ is Zariski dense. Then the $\mu-$stationary measure $\nu$ is a Rajchman measure, in other words
\begin{equation}
\widehat{ \nu}(k)\rightarrow 0\ \text{ as } |k|\rightarrow +\infty.
\end{equation}
\end{thm}
\begin{rem}
	Fourier decay of measures on fractal sets and its applications have been studied in \cite{kaufman1980continued},\cite{queffelec2003analyse},\cite{jordan2013fourier} and \cite{bourgain2017fourier}. Our situation is much general and we introduce a quite different method.
\end{rem}
Being a Rajchman measure is a local property (see \cite{kechris1987descriptive}): Indeed, let $\nu_1$ be a Rajchman measure. If $\nu_2$ is absolutely continuous with respect to $\nu_1$, then $\nu_2$ is also a Rajchman measure. Conversely, the sum of two Rajchman measures is a Rajchman measure. 

In this spirit, we have the following theorem:
\begin{thm}\label{thm:main}
Let $\mu$ be a Borel probability measure on $\slr$ with a finite exponential moment, and assume that the subgroup $\Gamma_{\mu}$ is Zariski dense. Let $\nu$ be the unique $\mu-$stationary measure. Assume that $r$ is a $C^1$ function on $\bb P^1$ and $\phi$ is a $C^2$ function on $\bb P^1$ such that $|\phi'|\geq 1/C_1>0$ on the support of $r$ and 
\[ \|r\|_{C^1},\|\phi\|_{C^2}\leq C_1 \text{ for some constant  }C_1>0. \]
Then we have
\begin{equation}
\int e^{i\xi\phi(x)}r(x)\dd \nu(x)\rightarrow 0\text{ as }|\xi|\rightarrow \infty,
\end{equation}
uniformly with respect to $C_1$.
\end{thm}
This is the main theorem of this paper. It will be proved in Section \ref{secdecfou}.
\begin{cor}\label{cor:main}
Let $\mu$ be a Borel probability measure on $\slr$ with a finite exponential moment, and assume that the subgroup $\Gamma_{\mu}$ is Zariski dense. Let $\nu$ be the unique $\mu-$stationary measure. Then for a $C^2-$diffeomorphism $\phi$ on $\bb P^1$, the pushforward of the stationary measure $\phi_*\nu$ is a Rajchman measure. In other words
\begin{equation}
\widehat{\phi_*\nu}(k)\rightarrow 0 \text{ as } |k|\rightarrow +\infty.
\end{equation}
\end{cor}
Theorem \ref{thm:fourier} is a special case of this corollary, where $\phi$ is the identity function.
\begin{proof}[Proof of Corollary \ref{cor:main} from Theorem \ref{thm:main}]
By the identification $\bb P^1\simeq \bb T$, we may consider all the objects as living on $\bb T$. Take a partition of unity of $\bb T$: let $r_1,r_2$ be non negative Lipschitz functions on $\bb T$ such that $r_1+r_2=1$, and the supports of $r_1,r_2$ are connected subintervals of $\bb T$. For $j=1,2,$ we can lift the function $\phi|_{\supp r_j}$ to a function $\phi_j$ from $\supp r_j$ to $\bb R$. Then
\begin{align*}
&\int_{\bb T}e^{2ik \phi(\theta)}\dd \nu(\theta)=\int_{\bb T}(r_1(\theta)+r_2(\theta))e^{2ik\phi(\theta)}\dd \nu(\theta)=\int_{\bb T}\left(e^{2ik \phi_1(\theta)}r_1(\theta)+e^{2ik\phi_2(\theta)}r_2(\theta)\right)\dd \nu(\theta).
\end{align*}
Since $\phi$ is a diffeomorphism, the functions $\phi_j,r_j$ satisfy the conditions in Theorem \ref{thm:main}. We use this theorem twice to conclude.
\end{proof}
Let us use another coordinate system on $\bb P^1$. We identify $\bb P^1$ with $\bb R\cup\{\infty \}$ through the map $\varphi(x)=v_1/v_2$, where $x=\bb Rv$ is a point in $\bb P^1$. Then the action of $\slr$ on $\bb P^1$ reads as the M\"obius action, that is for $r\in\bb R\cup\{\infty \}$ and $g=\begin{pmatrix}
a & b\\ c & d
\end{pmatrix}$ in $\slr$, we have $gr=\frac{ar+b}{cr+d}$.

If the support of a $\mu-$stationary measure $\nu$ does not contain $[1:0]$, then $\varphi_*\nu$ is a stationary measure on $\bb R$. From Theorem \ref{thm:main}, we get
\begin{cor}
Let $\mu$ be a Borel probability measure on $\slr$ with a finite exponential moment, and assume that the subgroup $\Gamma_{\mu}$ is Zariski dense. Let $\nu$ be the unique $\mu-$stationary measure. If the support of $\nu$ does not contain $[1:0]$, then the $\mu-$stationary measure $\varphi_*\nu$ is a Rajchman measure on $\bb R$. In other words
\begin{equation}
\widehat{\varphi_*\nu}(\xi)=\int_{\bb P^1}e^{i\varphi(x)\xi}\dd\nu(x)\rightarrow 0 \text{ as } |\xi|\rightarrow +\infty.
\end{equation}
\end{cor}
\begin{exa}[Solvable case]\label{exa:solvable}
For stationary measures on $\bb R$, consider the following
$$\mu=\frac{1}{2}(\delta_{g_1}+\delta_{g_2})=\frac{1}{2}\delta_{\begin{pmatrix}
	\sqrt{\lambda} & -1/\sqrt{\lambda}\\ 0 & 1/\sqrt{\lambda}
	\end{pmatrix}}+\frac{1}{2}\delta_{\begin{pmatrix}
	\sqrt{\lambda} & 1/\sqrt{\lambda} \\ 0 & 1/\sqrt{\lambda}
	\end{pmatrix}},$$
where $\lambda\in(0,1)$. Then the actions of $g_1,g_2$ are given by $g_1r=\lambda r-1,\,g_2r=\lambda r+1$ for $r\in\bb R$. By definition, a $\mu-$stationary measure $\nu$ on $\bb R$ must satisfy the equation
\begin{equation}\label{equ:stationary}
\nu=\mu*\nu=\frac{1}{2}((g_1)_*\nu+(g_2)_*\nu).
\end{equation}
Let $X_0,X_1,\dots$ be i.i.d. random variables such that $\bb P(X_0=1)=\bb P(X_0=-1)=1/2.$ Let $\nu_\lambda$ be the Bernoulli convolution with parameter $\lambda$, defined to be the distribution of $\sum_{j\geq 0}X_j\lambda^j$. The measure $\nu_\lambda$ satisfies \eqref{equ:stationary}, thus it is a $\mu$-stationary measure on $\bb R$. In \cite{erdos1939family}, Erd\"{o}s proved that when $\lambda^{-1}$ is a Pisot number, the Fourier transform of $\nu_{\lambda}$ does not converge to zero. In this example $\Gamma_{\mu}$ is solvable, so the Zariski density condition is necessary in the theorem.
\end{exa}
\begin{rem}
	1. A similar result for Bernoulli convolutions was obtained in \cite{kaufman76}. Kaufman proved that for Bernoulli convolutions $\nu_\lambda$, if $\lambda^{-1}$ is not a Pisot number, then it satisfies the same conclusion as in Corollary \ref{thm:main}. That is, the pushforward measure $\phi_*\nu_\lambda$ is a Rajchman measure, where $\phi$ is a $C^1$ function on $\bb R$ with $\phi'>0$ everywhere.
	
	\noindent2.Our result for the measure $\nu$ is stronger than being a Rajchman measure. Indeed, for a probability measure on $\bb T$, being a Rajchman measure is not invariant by diffeomorphisms. We can find examples in \cite{kaufman82}. A typical example is the standard $\frac{1}{3}-$Cantor measure $\nu$, which is not a Rajchman measure. Let $\phi$ be the quadratic map $r\mapsto r^2$. Then the pushforward measure $\phi_*\nu$ becomes a Rajchman measure with polynomial decay.
\end{rem}
One of our motivations for establishing Theorem \ref{thm:fourier} comes from the theory of Bernoulli convolutions. One of the main questions of this theory is to determine for which parameter $\lambda$, the measure $\nu_\lambda$ is absolutely continuous with respect to the Lebesgue measure. We have already mentioned that when $\lambda^{-1}$ is a Pisot number, Erd\"os proved that $\nu_\lambda$ is not a Rajchman measure. Thus, in particular, $\nu_\lambda$ is not absolutely continuous with respect to the Lebesgue measure. Recently, people have been interested in the same problem for stationary measures for random walks on $\slr$, see \cite{bourgain2012finitely},\cite{kaimanovich2011matrix}. Our result shows that we cannot generalize the method of Erd\"{o}s to the Zariski dense case.

Our other motivation is the same question for the Patterson-Sullivan measure on the limit set of Fuchsian groups. With Theorem \ref{thm:fourier}, it suffices to prove that there exists a probability measure $\mu$ on $\slr$ such that the Patterson-Sullivan measure is $\mu$-stationary, and $\mu$ has a finite exponential moment.

In \cite{lalley1989renewal} and \cite{lalley1986gibbs}, Lalley announced the existence of such a $\mu$ for Schottky groups. But Lalley's proof only works for Schottky semigroups. In \cite{connell2007harmonicity}, the authors proved the existence of such a $\mu$ without the moment condition in geometrically finite cases. Combining the methods of Connell, Muchnik and Lalley, we can prove the existence of such a measure $\mu$ for convex cocompact Fuchsian groups, see \cite{li2017exponential}. Therefore, we have
\begin{cor}\label{cor:ps}
Let $\Gamma$ be a convex cocompact Fuchsian group. Then the Patterson-Sullivan measure associated to $\Gamma$ is a Rajchman measure.
\end{cor}
\begin{rem}
	Corollary \ref{cor:ps} also holds if we replace the Patterson-Sullivan measure by any Gibbs measure. In \cite{li2017exponential}, we have a similar realization for any Gibbs measure associated to a convex cocompact Fuchsian group, as it is done by Lalley for any Gibbs measure on the limit set of a Schottky semigroup in \cite{lalley1986gibbs}.
\end{rem}
\begin{rem}\label{rem:unispe}
	Using the uniform spectral gap proved in \cite{naud2005expanding}, we can prove a polynomial decay in the convergence to zero of the Fourier coefficients of the $\mu$-stationary measure, when the support of $\mu$ is the set of generators of a Schottky semigroup. In this case, the uniform spectral gap implies an exponential error term in the renewal theorem, which is the only obstacle for polynomial decay. Please see Remark \ref{rem:poly} for more details.
	We believe it is true for the general case, but the question is still open.
\end{rem}
\begin{rem}
	Very recently, Bourgain and Dyatlov \cite{bourgain2017fourier} have proved a polynomial decay of the Fourier coefficients of the Patterson-Sullivan measure associated to a convex cocompact Fuchsian group. Their method, which comes from additive combinatorics, is totally different from ours. They use the Fourier decay bound and the fractal uncertainty principle to obtain an essential spectral gap for a convex cocompact hyperbolic surface. We can not recover their result directly as in Remark \ref{rem:unispe}. It is possible if we modifier some steps and use the uniform spectral gap in \cite{naud2005expanding}, but we do not pursue in this direction in this work. 
\end{rem}
On the other hand, in the geometrically finite case, this approach can not work. The finite exponential moment condition is impossible for noncompact lattice $\Gamma$ in $\slr$ (see \cite{guivarch1993winding}, \cite{deroin2009circle}, \cite{blachere2011harmonic}). That is, if $\mu$ is a measure on $\Gamma$ with a finite first moment, then the $\mu$-stationary measure $\nu$ is singular with respect to the Lebesgue measure. Maybe the generalization of the method of \cite{jordan2013fourier} works in this case, where they proved the Gibbs measures for the Gauss map which has dimension greater than $1/2$ are Rajchman measures.

In this paper, our main idea is to obtain the convergence to zero of Fourier coefficients from a renewal type result.

\textbf{The strategy of proof}: 
To simplify, identify $\bb P^1$ with $\bb T=\bb R/\pi\bb Z$ as before. The starting point is the relation $\nu=\mu*\nu$.
Consider a random walk on $\slr$, $X_n=b_1b_2\cdots b_n$, where $b_j$ are independent random variables with the same law $\mu$. Let $\mathscr B_n$ be the Borel $\sigma-$algebra generated by $X_1,\dots,X_n$. Let $Y_n=(X_n)_*\nu$. They are random variables which take values in the space of Borel measures on $\bb T$. By definition, we have
\[\bb E(Y_{n+1}|\mathscr B_n)=\bb E((X_n)_*(b_{n+1})_*\nu|\mathscr B_n)=(X_n)_*\bb E((b_{n+1})_*\nu)=Y_n. \]
Therefore $\{Y_n\}$ is a martingale. For $t>0$, we define the stopping time by $\tau=\inf\{n\in\bb N|\log\|X_n\|\geq t \}$. Then the martingale property implies that 
\[ \bb E((X_{\tau})_*\nu)=\bb E(Y_\tau)=\bb E(Y_0)=\nu. \]
(See Proposition \ref{prop:stoptime}). Thus for the Fourier coefficients, we have for $k\in 2\bb Z$ (since $\overline{\hat{\nu}(k)}=\hat{\nu}(-k)$, we only consider $k\geq 0$.)
\[\hat{\nu}(k)=\int e^{ikx}\dd\nu(x)=\int e^{ikx}\dd \bb E((X_\tau)_*\nu)(x)=\int \bb E(e^{ikX_\tau x})\dd\nu(x). \]
Recall our circle $\bb T$ is $\bb R/\pi\bb Z$. The idea is to find some cancellations in the ``trigonometric series" $\bb E(e^{ik X_\tau x})$. By the Cauchy-Schwarz inequality, it suffices to prove $\bb E(e^{ik(X_\tau x-X_\tau y)})\rightarrow 0$ as $k\rightarrow\infty$. 

By analogy with the case of classical random walks on $\bb R$, we expect that there exists a measurable density function $p$ on $\bb R^+$ such that for a continuous compactly supported function $f$ on $\bb R$ and $t\in\bb R$,
\[\bb E(f(\log\|X_\tau\|-t))\longrightarrow\int_{\bb R^+}f(u)p(u)\dd u\text{ as }t\rightarrow+\infty. \]
Then absolute continuity of the limit distribution would imply the convergence to zero of $\hat{\nu}(k)$.

In the actual proof, we do not use this stopping time, but a residue process. Indeed, the latter is easier to treat with transfer operators and Fourier analysis. We will establish a limit theorem for the residue process, a generalization of the renewal theorem, in Section \ref{secrentheory}.

\textbf{Notation}: When $f$ and $g$ are functions on a set $X$, we write $f(x)\lesssim g(x)$, if there exists $C>0$ independent of $x\in X$ such that $f(x)\leq Cg(x)$, and $f(x)=O( g(x))$ means $|f(x)|\lesssim g(x)$. We also write $f(x,y)=O_y(g(x,y))$, which means $|f(x,y)|\leq C_yg(x,y)$, where $C_y$ is a constant only depending on $y$.

We introduce a notation $\oexpe(s)$. We write $f(\epsilon,s)=\oexpe(s)$ if for $\epsilon>0$ and $s\in\bb R$, there exists a constant $\epsilon'>0$ such that $f(\epsilon,s)=O(e^{-\epsilon' s})$, where all the constants only depend on $\epsilon$. We write $f(s)=\oexp(s)$, if there exists a uniform constant $\epsilon'$ such that $f(s)=O(e^{-\epsilon's})$.
\section*{Acknowledgments}
The author wishes to express his gratitude to Jean-François Quint for suggesting the problem and for many stimulating conversations. The author is grateful to the referee for carefully reading the manuscript and for helpful remarks.
\section{Preliminaries on random walks on $\bb P^1$}
\label{sec:prelim}
Fix the norm induced by the standard inner product on $\bb R^2$, $\|v\|=\sqrt{v_1^2+v_2^2}$, which is $\so(\bb R)$ invariant. Then define a metric on $\bb P^1$. For two points $x=\bb Rv,\ x'=\bb Rw$, we set
\begin{equation*}
d(x,x')=\frac{|\det(v,w)|}{\|v\|\|w\|}.
\end{equation*} 
This is a sine distance. If we write $x=\bb R\begin{pmatrix}
\cos \theta\\\sin\theta
\end{pmatrix}$ and $x'=\bb R\begin{pmatrix}
\cos\theta'\\\sin\theta'
\end{pmatrix}$, then $d(x,x')=|\sin(\theta-\theta')|$.
From now on, we write $G=\slr$ and $X=\bb P^1$.
\begin{defi}
For $g$ in $G$ and $x=\bb Rv$ in $X$, define the function $\sigma:G\times X\rightarrow \bb R$ by 
$\sigma(g,x)=\log\frac{\|gv\|}{\|v\|}$.
\end{defi}
This function $\sigma$ is a cocycle, because for $g,h$ in $G$ we have
\[\sigma(gh,x)=\log\frac{\|ghv\|}{\|v\|}=\log\frac{\|ghv\|}{\|hv\|}+\log\frac{\|hv\|}{\|v\|}=\sigma(g,hx)+\sigma(h,x),\]
where we use the fact that the action is linear, $hx=\bb Rhv$.
\begin{lem}\label{lem:distance} For $g$ in $G$ and $x,x'$ in $X$ with $x\neq x'$, we have
\begin{equation}\label{equ:distance}
\frac{d(gx,gx')}{d(x,x')}=\exp(-\sigma(g,x)-\sigma(g,x')).
\end{equation}
\end{lem}
\begin{proof}
As in the definition of the distance $d(\cdot,\cdot)$, we take two non zero vectors $v$ and $w$ in $x$ and $x'$ respectively. By definition,
\[\frac{d(gx,gx')}{d(x,x')}=\left|\frac{\det(gv,gw)}{\det(v,w)}\right|\frac{\|v\|\|w\|}{\|gv\|\|gw\|}=\left|\det g\right|\frac{\|v\|\|w\|}{\|gv\|\|gw\|}=\exp\left(-\sigma(g,x)-\sigma(g,x')\right). \]
The proof is complete.
\end{proof}
If the point $x$ is near $x'$, we know from the above equation that the cocycle $\sigma$ is essentially the logarithm of the contracting or expanding ratio. Let $\mu$ be a Borel probability measure on $G$, and let $b_1,b_2,\cdots$ be independent random variables with the same law $\mu$. Then  
the behavior of the mean value of the cocycle, 
$$\frac{1}{n}\sigma(b_nb_{n-1}\cdots b_1,x)=\frac{1}{n}\left(\sum_{j=1}^{n}\sigma(b_j,b_{j-1}\cdots b_1x) \right),$$ 
follows an asymptotic law similar to the law of large numbers. In particular,
\begin{thm}\label{lawln}\cite{furstenberg1963noncommuting}\cite{guivarc1985frontiere}
Let $\mu$ be a Borel probability measure on $G$ having an exponential moment. Assume that the subgroup $\Gamma_{\mu}$ is Zariski dense. Then for all $x$ in $X$, random variables $b_j$ defined as above, we have
\begin{equation}
\lim_{n\rightarrow\infty}\frac{\sigma(b_nb_{n-1}\cdots b_1,x)}{n}=\int_{G}\int_{X}\sigma(g,y)\dd\mu(g)\dd\nu(y)=\sigma_{\mu}>0\ \ \text{a.s }\mu^{\otimes \bb N^*}.
\end{equation}
\end{thm}
The constant $\sigma_{\mu}$ is called the Lyapunov exponent of $\mu$. 
\begin{thm}[H\"older regularity]\cite{guivarc1990produits}\cite[Chapter 6,Proposition 4.1]{bougerol1985products}
	 Under the assumptions of Theorem \ref{thm:fourier}, there exist constants $C>0,\,\alpha>0$ such that for every $x$ in $X$ and $r>0$ we have
	\begin{equation}\label{ineq:regsta}
	\nu(B(x,r))\leq Cr^{\alpha}.
	\end{equation}
\end{thm}
We need the Cartan decomposition of the Lie group $G$, i.e. $G=\so A^+\so $, where $A^+=\{\begin{pmatrix}
e^t & 0 \\ 0 & e^{-t} 
\end{pmatrix}, t\geq0\}$. For $g$ in $G$, we can write $g=k_ga_gl_g$, where $k_g$, $l_g$ are in $\so $, and $a_g=diag\{e^{\kappa(g)},e^{-\kappa(g)} \}$ is the diagonal matrix whose diagonal elements are $e^{\kappa(g)}$ and $e^{-\kappa(g)} $ with $\kappa(g)\geq0$. The positive number $\kappa(g)$ is called the Cartan projection. Identify the two spaces $X$ and $\bb T\simeq \bb R/\pi\bb Z$. For an element $x$ in $X$, associate it to the unique element $\theta(x)$ in $\bb R/\pi\bb Z$ satisfying $x=\bb R\begin{pmatrix}
\cos  \theta(x) \\ \sin \theta(x)
\end{pmatrix}$. When there is no ambiguity, we will abbreviate $\theta(x)$ to $x$.

Let $e_1=\bb R\begin{pmatrix}
1 \\ 0
\end{pmatrix},\,e_2=\bb R\begin{pmatrix}
0 \\ 1
\end{pmatrix}$,  which mean elements in $X$. Let $r_\theta=\begin{pmatrix}
\cos\theta & \sin\theta\\
-\sin\theta & \cos\theta
\end{pmatrix}$ be a rotation matrix in $G$. For $g$ in $G$, choosing a decomposition $g=k_ga_gl_g$, we define $x_g^m=l_g^{-1}e_2, x_g^M=k_ge_1$. If $\kappa(g)>0$, then $x_g^M,x_g^m$ are uniquely defined.
\begin{prop}
For $g$ in $G$ with $\kappa(g)>0$, we have
\begin{equation}\label{equ:xgmxginv}
x_g^m=x_{g^{-1}}^M.
\end{equation}
\end{prop}
\begin{proof} 
For a real number $a\neq 0$, we have
\[\begin{pmatrix}
a & 0 \\ 0 & a^{-1}
\end{pmatrix}=\begin{pmatrix}
0 & 1 \\ -1 & 0
\end{pmatrix}\begin{pmatrix}
a^{-1} & 0 \\ 0 & a
\end{pmatrix}\begin{pmatrix}
0 & -1 \\ 1 & 0
\end{pmatrix}=r_{\pi/2}\begin{pmatrix}
a^{-1} & 0\\ 0 & a
\end{pmatrix}r_{3\pi/2}. \]
This implies that
\[g^{-1}=(k_ga_gl_g)^{-1}=l_g^{-1}a_g^{-1}k_g^{-1}=l_g^{-1}r_{\pi/2}a_gr_{3\pi/2}k_g^{-1}. \]
Therefore $x_{g^{-1}}^M=l_g^{-1}r_{\pi/2}e_1=l_g^{-1}e_2=x_g^m$.
\end{proof}
\begin{lem}\label{lem:coccardis}
For $g$ in $G$ and $x=\bb Rv$ in $X$, we have
\begin{equation}\label{ineq:coccardis}
d(x_g^m,x)\leq \frac{\|gv\|}{\|g\|\|v\|}\leq d(x_g^m,x)+e^{-2\kappa(g)}.
\end{equation}
Another form that will be used frequently is 
\[\sigma(g,x)\geq \kappa(g)+\log d(x_g^m,x). \]
\end{lem}
\begin{proof} Suppose that the vector $v$ has norm 1, then
	\begin{align*}
	\frac{\|gv\|}{\|v\|}=\frac{\|k_ga_gl_gv\|}{\|v\|}=\frac{\|a_gl_gv\|}{\|l_gv\|}.
	\end{align*}
	Since $d(x_g^m,x)=d(l_g^{-1}e_2,x)=d(e_2,l_gx)$, it suffices to prove this inequality for diagonal elements, in other words
	$g=diag\{e^{\kappa(g)},e^{-\kappa(g)} \}$. Hence
	\begin{align}\label{equ:cocycle}
	\frac{\|gv\|}{\|v\|}=\left|\begin{pmatrix}
	e^{\kappa(g)} & 0\\0 & e^{-\kappa(g)}
	\end{pmatrix}\begin{pmatrix}
	v_1 \\ v_2
	\end{pmatrix}\right|
	=|e^{2\kappa(g)}v_1^2+e^{-2\kappa(g)}v_2^2|^{1/2}.
	\end{align}
	The equality $d(x_g^m,x)=d(e_2,x)=|v_1|$  implies that
	\begin{align*}
	&\frac{\|gv\|}{\|g\|\|v\|}\geq |v_1|=d(x_g^m,x),\\
	&\frac{\|gv\|}{\|g\|\|v\|}\leq |v_1|+e^{-2\kappa(g)}=d(x_g^m,x)+e^{-2\kappa(g)}.
	\end{align*}
	The proof is complete.
\end{proof}
The following lemma is an important tool, which gives a precise approximation of the cocycle by the Cartan projection and distance.
\begin{lem}\label{lem:carcoc}
	Let $x,x'$ be two points in $X$ and let $g$ be in $G$. Assume that  $$e^{-2\kappa(g)}+d(x_g^m,g^{-1}x)\leq \frac{1}{2}d(g^{-1}x',x),$$
	then
	\begin{equation}\label{ineq:carcoc}
		|\sigma(g,x)-\kappa(g)-\log d(g^{-1}x',x)|\leq 2\frac{e^{-2\kappa(g)}+d(x_g^m,g^{-1}x')}{d(g^{-1}x',x)}.
	\end{equation}
\end{lem}
\begin{proof}
 Inequality \eqref{ineq:coccardis} implies that
	\begin{align*}
	|e^{\sigma(g,x)-\kappa(g)}-d(g^{-1}x',x)|&\leq \max\{|d(x_g^m,x)-d(g^{-1}x',x)|,|e^{-2\kappa(g)}+d(x_g^m,x)-d(g^{-1}x',x)| \}\\
	&\leq e^{-2\kappa(g)}+d(x_g^m,g^{-1}x').
	\end{align*}
	Thus by hypothesis, we have
	\[|\exp(\sigma(g,x)-\kappa(g))-d(g^{-1}x',x)|\leq 1/2d(g^{-1}x',x).\]
	Since $|\log(1+t)|\leq 2|t|$ for $t>-1/2$, we obtain
	\begin{align*}
	&|\sigma(g,x)-\kappa(g)-\log d(g^{-1}x',x)|=\log|1+\frac{\exp(\sigma(g,x)-\kappa(g))-d(g^{-1}x',x)}{d(g^{-1}x',x)}| \\
	&\leq 2\frac{|\exp(\sigma(g,x)-\kappa(g))-d(g^{-1}x',x)|}{d(g^{-1}x',x)}\leq 2\frac{e^{-2\kappa(g)}+d(x_g^m,g^{-1}x')}{d(g^{-1}x',x)}.
	\end{align*}
	The proof is complete.
\end{proof}

In the next proposition we summarize the large deviations principle for the cocycle and for the Cartan projection,
\begin{prop}\cite[Thm13.11, Thm 13.17]{benoistquint}
Under the assumptions of Theorem \ref{thm:fourier}, 
for every $\epsilon>0$ we have
\begin{align}\label{ineq:lardev}
&\mu^{*n}\{g\in G|\ |\sigma(g,x)-n\sigma_{\mu}|\geq n\epsilon \}= \oexpe(n),\\
&\label{ineq:lardev1}\mu^{*n}\{g\in G|\ |\kappa(g)-n\sigma_{\mu}|\geq n\epsilon \}=\oexpe(n),
\end{align}
uniformly for all $x$ in $X$ and $n\geq 1$.
\end{prop}
Let $t$ be a real number. Write $[t]$ for the integer part of $t$. 
\begin{cor}\label{cor:lardev}
Under the assumptions of Theorem \ref{thm:fourier}, for every $\epsilon>0$ we have
\begin{align*}
&\sum_{m\geq n}\mu^{*m}\{g\in G|\sigma(g,x)\leq t \}=\oexpe(n),\\
&\sum_{m\geq n}\mu^{*m}\{g\in G|\kappa(g)\leq t \}=\oexpe(n),
\end{align*}
uniformly for all $x$ in $X$, $t>0$ and $n\geq [\frac{t}{\sigma_{\mu}-\epsilon}]$.
\end{cor}
By the hypothesis of finite exponential moment and the Chebyshev inequality, we have
\begin{lem}Under the assumptions of Theorem \ref{thm:fourier}, 
let $M_\mu$ be the finite exponential moment of $\mu$ defined by $M_\mu=\int \|g\|^{\epsilon_1}\dd\mu(g)$. For $s>0$, we have
\begin{equation}\label{ineq:moment}
\mu\{g\in G|\kappa(g)\geq s \}\leq M_\mu e^{-\epsilon_1s}. 
\end{equation}
\end{lem}
\begin{cor}\label{cor:lardev1}
Under the assumptions of Theorem \ref{thm:fourier}, for every $\epsilon>0$ we have
\begin{align*}
&\sum_{m\leq n}\mu^{*m}\{g\in G|\sigma(g,x)\geq t \}=\oexpe(t), \\
&\sum_{m\leq n}\mu^{*m}\{g\in G|\kappa(g)\geq t \}=\oexpe(t),
\end{align*}
uniformly for all $x$ in $X$, $t>0$ and $n= [\frac{t}{\sigma_{\mu}+\epsilon}]$.
\end{cor}
\begin{proof}
	The inequality about the cocycle follows from the one about the Cartan projection, because $\kappa(g)\geq\sigma(g,x)$. It suffices to prove the second inequality: 
	\begin{itemize}
		\item When $m\leq \epsilon_2t$, where $\epsilon_2>0$ is a small constant such that 
		$\epsilon_2\leq \epsilon_1/(2\log M_\mu)$, from Chebyshev's inequality and the subadditivity of the Cartan projection, we have
		\begin{align*}
		\sum_{m=1}^{[\epsilon_2t]}\mu^{\otimes m}\{\kappa(g)\geq t \}&\leq 	\sum_{m=1}^{[\epsilon_2t]}e^{-\epsilon_1 t}\int e^{\epsilon_1\kappa(g)}\dd\mu^{\otimes m}(g)\leq \sum_{m=1}^{[\epsilon_2t]}e^{-\epsilon_1 t}\left|\int\|g\|^{\epsilon_1}\dd\mu(g)\right|^m\\
		&\leq e^{-\epsilon_1 t}M_\mu^{[\epsilon_2t]}/(M_\mu-1).
		\end{align*}
		This implies that $\sum_{m=1}^{[\epsilon_2t]}\mu^{\otimes m}\{\kappa(g)\geq t \}\lesssim e^{-t\epsilon_1/2}$. 
		
		\item When $m\in[\epsilon_2t,t/(\sigma_\mu+\epsilon)]$, 
		we have $\kappa(g)>t\geq m(\sigma_\mu+\epsilon)$. Then use \eqref{ineq:lardev1} to deduce that the measure of this part is less than $\sum_{m\in[\epsilon_2t,t/(\sigma_\mu+\epsilon)]}\oexpe(m)=\oexpe(t)$.
	\end{itemize}
	The proof is complete.
\end{proof}
The following proposition describes regularity properties of $\mu^{*n}$, which is a corollary of the large deviations principle.
\begin{prop}\cite[Prop14.3]{benoistquint}
Under the assumptions of Theorem \ref{thm:fourier}, for every $\epsilon>0$ we have
	\begin{align}
		\label{ineq:regcon}&\mu^{*n}\{g\in G|\ d(gx,x')\leq e^{-n\epsilon} \}=\oexpe(n),\\
		\label{ineq:xgmx}&\mu^{*n}\{g\in G|\ d(x^M_g,x)\leq e^{-n\epsilon} \}=\oexpe(n),\\
		\label{ineq:xgmgx}&\mu^{*n}\{g\in G|\ d(x_g^m,g^{-1}x)\geq e^{-(2\sigma_\mu-\epsilon)n} \}=\oexpe(n),\\
		\label{ineq:xgminv}&\mu^{*n}\{g\in G|\ d(x_g^M,gx)\geq e^{-(2\sigma_\mu-\epsilon)n} \}=\oexpe(n),
	\end{align}
uniformly for all $x,\,x'$ in $X$ and $n\geq 1$.
\end{prop}
\begin{cor}\label{cor:regcon}
Under the assumptions of Theorem \ref{thm:fourier}, for every $\epsilon>0$ we have
\begin{align}
	\label{ineq:regcon1}\mu^{*n}\{g\in G|\ d(gx,x')\leq e^{-t} \}=\oexpe(t),
\end{align}
uniformly for all $x,x'$ in $X$, $t>0$ and $n\geq t/\epsilon$.

For every $\epsilon>0$ we have 
\begin{equation}
\label{ineq:xgminv1}\mu^{*n}\{g\in G|\ d(x_g^M,gx)\geq e^{-t} \}=\oexpe(t),
\end{equation}
uniformly for all $x$ in $X$, $t>0$ and $n\geq t/(2\sigma_\mu-\epsilon)$.
\end{cor}
\begin{proof}
There exists an integer $n_t\leq n$ such that $\epsilon n_t<t\leq \epsilon(n_t+1)$. 
	By inequality \eqref{ineq:regcon}, we have $\mu^{* n_t}\{d(gx,x')\leq e^{-\epsilon n_t} \}\lesssim e^{-\epsilon' n_t}$. This implies that
	\begin{align*}
	\mu^{* n}\{g\in G|d(gx,x')\leq e^{-t} \}&=\int_{G}\mu^{* n_t}\{l\in G|d(l(hx),x')\leq e^{- t} \}\dd\mu^{* (n-n_t)}(h)\\
	&\leq \int_{G}\mu^{* n_t}\{l\in G|d(l(hx),x')\leq e^{- \epsilon n_t} \}\dd\mu^{* (n-n_t)}(h)\\
	&\lesssim e^{-\epsilon' n_t}\lesssim e^{-\epsilon'(t/\epsilon-1)}\lesssim e^{-\epsilon't/\epsilon }.
	\end{align*}
The second inequality follows from the same argument. 
\end{proof}
The following lemma describes the difference between the cocycle and the Cartan projection.
\begin{lem}\cite[Lemma 17.8]{benoistquint}\label{lem:coccar}
Under the assumptions of Theorem \ref{thm:fourier}, for every $\epsilon>0$, there exist $C>0,\epsilon'>0$ such that for all $n\geq l>0$ and $x$ in $X$, there exists a subset $S_{n,l,x}\subset G\times G$, which satisfies
	\[\mu^{*(n-l)}\otimes\mu^{*l}(S_{n,l,x}^c)\leq Ce^{-\epsilon'l}=\oexpe(l), \]
	and for all $(g_1,g_2)\in S_{n,l,x}$, we have
	\[|\kappa(g_1g_2)-\sigma(g_1,g_2x)-\kappa(g_2)|\leq e^{-\epsilon l}. \]
\end{lem}

By the identification $X\simeq\bb T$, we can work on $\bb T$. Since the circle $\bb T$ is a quotient space of $\bb R$, it has the induced orientation. For two different points $x,y$ in $\bb T$, which are not the two endpoints of a diameter, they divide the circle into two arcs. Call the arc with longer length the large arc, and the other arc the small arc $x\smallfrown y$. For a function $\phi$ on $\bb T$, it can be seen as a function $\Phi$ on $\bb R$ with period $\pi$. Define $\phi'(\theta)$ as the derivative of $\Phi$.

We introduce a sign for two different points $x,y$ in $X$, where $x,y$ are not the two endpoints of a diameter. If in the small arc $ x\smallfrown y$, the point $x$ is the start point in the orientation sense, then we define $\sgn(x,y)=1$; otherwise, we define $\sgn(x,y)=-1$. We have a Newton-Leibniz formula on the circle
\begin{equation}\label{equ:newcir}
\phi( y)-\phi( x)=\sgn(x,y)\int_{ x\smallfrown  y}\phi'(\theta)\dd\theta,
\end{equation}
where $\dd\theta$ is the Lebesgue measure on $\bb T$ induced by the Lebesgue measure on $\bb R$ with total mass $\pi$.
\begin{defi}[Orientation]
Let $x,y,z$ be three points in $X$. Define
\begin{equation*}
\sgn(x,y,z)=\left\{ \begin{aligned}
0 &\text{ if any two points coincide},\\
1 &\text{ if }\{ x, y, z\} \text{ is counterclockwise},\\
-1 &\text{ otherwise}.
\end{aligned}
\right.
\end{equation*} 
\end{defi}
\begin{prop}\label{prop:sgngxy}
Let $x,y$ be two different points in $X$, and let $g$ be in $G$ such that $\kappa(g)>2$ and $d(x_g^m,x),d(x_g^m,y)>e^{-\kappa(g)}$. Then
\begin{equation}\label{equ:sgngxy}
\sgn(gx,gy)=\sgn(x,y,x_g^m).
\end{equation}
\end{prop}
\begin{proof}
With the same argument as in the proof of Lemma \ref{lem:coccardis}, it suffices to prove the statement in case $g=a_g$, that is $\sgn(a_gx,a_gy)=\sgn(x,y,e_2)$.

If $x'$ is a point in $X$ such that $d(e_2,x')>e^{-\kappa(g)}$, then
\[d(a_gx',e_1)=d(a_gx',a_ge_1)=d(x',e_1)\exp(-\sigma(a_g,x')-\sigma(a_g,e_1)). \]
By \eqref{ineq:coccardis}, we obtain $\sigma(a_g,x')\geq\kappa(g)+\log d(e_2,x')>0$, so 
\[d(a_gx',e_1)\leq \exp(-\kappa(g))\leq e^{-2}. \]
Thus the action of $a_g$ on the interval $B(e_2,e^{-\kappa(g)})^c$ is contracting with fixed point $e_1$, and the image is in the interval $B(e_1,e^{-\kappa(g)})$. Especially, $e_2$ is not in $B(e_1,e^{-\kappa(g)})$ and the small arc $ a_gx\smallfrown a_gy$ is contained in $B(e_1,e^{-\kappa(g)})$. By definition we have
\[\sgn(a_gx,a_gy)=\sgn(a_gx,a_gy,e_2).\]
Since the action of $a_g$ on $\bb T$ preserves the orientation, we have $\sgn(a_gx,a_gy,e_2)=\sgn(x,y,e_2)$. The proof is complete.
\end{proof}
\section{Decrease of the Fourier transform}
\label{secdecfou}
Here we give a proof of Theorem \ref{thm:main}, by admitting the technical results that will be proved in the following two sections. Recall the notations $G=\slr$ and $ X=\bb P^1$.

\begin{defi}\label{def:Sigma}
Let $\Sigma=\bigcup_{n\in\bb N}G^{\times n}$ be the symbol space of all finite sequences with elements in $G$. Let $\mu$ be a Borel probability measure on $G$, and let $\mu^{\otimes n}$ be the product measure on $G^{\times n}$. Then $\mu^{\otimes n}$ can be seen as a measure on $\Sigma$ which is nonzero only on $G^{\times n}$. Let $\bar{\mu}$ be the measure on $\Sigma$ defined by $\bar{\mu}=\delta_{\emptyset}+\mu+\mu^{\otimes 2}+\cdots$. 

Let the integer $\omega(g)$ be the length of an element $g$ in $\Sigma$. Then an element $g$ can be written as $(g_1,g_2,\dots,g_{\omega})$, where $\omega$ is the abbreviation of $\omega(g)$. 

Let $T$ be the shift map on $\Sigma$, defined by $Tg=T(g_1,g_2,\dots,g_{\omega})=(g_1,g_2,\dots,g_{\omega-1})$, when $\omega(g)\geq 2$,  and $Tg=\emptyset$, when $\omega(g)=1,0$.

Let $L$ be the left shift map on $\Sigma$, defined by $Lg=L(g_1,g_2,\dots,g_{\omega})=(g_2,\dots,g_{\omega-1},g_{\omega})$, when $\omega(g)\geq 2$, and $Lg=\emptyset$, when $\omega(g)=1,0$.

When considering the action of $g$ on $X$, we write $gx=g_1\cdots g_\omega x$, $\sigma(g,x)=\sigma(g_1\cdots g_\omega,x)$, $x_g^m=l^{-1}_{g_1\cdots g_\omega}e_2$, as well as the Cartan projection $\kappa(g)=\kappa(g_1\cdots g_\omega)$.
\end{defi}

\begin{rem}
When using this definition, we may meet the convolution measure $\mu^{*n}$ on $G$ or the product measure $\mu^{\otimes n}$ on $G^{\times n}$. Denote $F:G^{\times n}\rightarrow G$ by $F(g_1,g_2,\dots,g_n)=g_1\cdots g_n$, then $F_*(\mu^{\otimes n})=\mu^{*n}$.
\end{rem}

\begin{defi}
For $t>0$, define two sets that contain all the sequences which make the value of the Cartan projection pass $t$, 
$$M_t^+=\{g\in\Sigma|\ \kappa(Tg)< t\leq \kappa(g) \},\ M_t^-=\{g\in\Sigma|\ \kappa(Tg)\geq t>\kappa(g) \}.$$
\end{defi}

\begin{rem}\label{rem:ransto}
In some special cases, for $b_j$ in $\supp\mu$, the Cartan projection $\kappa(b_1b_2\cdots b_n)$ is increasing with respect to $n$. Then $M_t^-$ has $\bar{\mu}$ measure zero. Let $X_n=b_1b_2\cdots b_n$ be a random walk on $G$, where $b_j$ are i.i.d. random variables taking values in $G$ with the same law $\mu$.
Let $\tau$ be the stopping time defined by ${\tau}=\inf\{n\in\bb N|\kappa(X_{n})\geq t\}$. In such special case
\begin{align*}
\bar{\mu}(M_t^+\cap G^{\times n})=\bb P({\tau}=n).
\end{align*}
So in the measure sense, $M_t^+$ is a set of the steps. That is for $\bar{\mu}-$almost every $g$ in $M_t^+$, it is of the form $g=(b_1,b_2,\dots,b_\tau)=(X_1,X_1^{-1}X_2,\dots,X_{\tau-1}^{-1}X_\tau)$ which corresponds to the set of steps of the trajectory $(X_1,X_2,\dots,X_\tau)$. But this is not always true for general cases.
\end{rem}
By Corollary \ref{cor:lardev}, these two sets $M_t^+,\ M_t^-$ have finite $\bar{\mu}$ measure. We have a property of $M_t^+,\ M_t^-$ due to the definition of stationary measures. Our proof is a generalization of the property of the stopping time for martingales.

\begin{prop}\label{prop:stoptime}
Under the assumptions of Theorem \ref{thm:fourier}, for a real number $t>0$ and a continuous function $f$ on $X$, we have
\begin{align*}
\int_X f(x)\dd\nu(x)=\int_X\left(\int_{g\in M_t^+}f(gx)\dd\bar\mu(g)-\int_{g\in M_t^-}f(gx)\dd\bar\mu(g)\right)\dd\nu(x).
\end{align*}
\end{prop}

\begin{proof}
For a natural number $N$, let
\begin{align*}
F_N&=\int_X\left(\int_{g\in M_t^+,\omega(g)\leq N}f(gx)\dd\bar{\mu}(g)-\int_{g\in M_t^-,\omega(g)\leq N}f(gx)\dd\bar{\mu}(g)\right.\\
&\left.+\int_{\omega(g)=N,\kappa(g)< t}f(gx)\dd\bar{\mu}(g)\right)\dd\nu(x).
\end{align*}
Then $F_o=\int_Xf(x)\dd\nu(x)$. Since all the terms are finite, we have
\begin{align*}
F_{N+1}-F_N&=\int_X\left(\int_{g\in M_t^+,\omega(g)= N+1}f(gx)\dd\bar{\mu}(g)-\int_{g\in M_t^-,\omega(g)= N+1}f(gx)\dd\bar{\mu}(g)\right.\\
&\left.+\int_{\omega(g)=N+1,\kappa(g)< t}f(gx)\dd\bar{\mu}(g)-\int_{\omega(g)=N,\kappa(g)< t}f(gx)\dd\bar{\mu}(g)\vphantom{\int_1}\right)\dd\nu(x).
\end{align*}
By the relation $\nu=\mu*\nu$, the set of integration of the last term becomes $\{\omega(g)=N+1,\kappa(Tg)< t \}$.
Compare these sets of integration 
\begin{align*}
\{g\in & M_t^+,\omega(g)= N+1\}\cup \{\omega(g)=N+1,\kappa(g)< t \}\\
=&\{\omega(g)= N+1,\kappa(Tg)< t,\kappa(g)\geq t \}\cup \{\omega(g)=N+1,\kappa(g)< t \}\\
=&\{\omega(g)= N+1,\kappa(Tg)\geq t,\kappa(g)< t \}\cup\{\omega(g)=N+1,\kappa(Tg)< t\}\\
=&\{g\in M_t^-,\omega(g)= N+1\}\cup\{\omega(g)=N+1,\kappa(Tg)< t\}.
\end{align*}
Therefore, $F_{N+1}=F_N=\cdots=F_o$. Corollary \ref{cor:lardev} and Inequality \eqref{ineq:lardev1} imply that $\bar\mu\{g\in M_t^\pm,\omega(g)> N \}$, $\bar\mu\{\omega(g)=N,\kappa(g)< t \} \rightarrow 0$, as $N\rightarrow \infty$.
Thus
\[F_N\rightarrow \int_X\left(\int_{g\in M_t^+}f(gx)\dd\bar\mu(g)-\int_{g\in M_t^-}f(gx)\dd\bar\mu(g)\right)\dd\nu(x) \text{ as } N\rightarrow\infty,\]
which completes the proof.
\end{proof}

With these preparations, we start to prove Theorem \ref{thm:main}, by admitting Lemma \ref{lem:ntfin}, Corollary \ref{cor:ntsma} and Proposition \ref{prop:rescar}. 
\begin{proof}[Proof of Theorem \ref{thm:main}]
We will prove that there exist constants $\epsilon_0>0, C_0>0$ such that for every $s>0$, the Fourier transform $\int e^{i\xi\phi(\theta)}r(\theta)\dd \nu(\theta)$ is less than $C_0e^{-\epsilon_0 s}$ for all $|\xi|$ large enough depending on $s$. 

Fix a constant $\epsilon_3\leq 1/10$. Write $t=(\log|\xi|-s)/2$, and take $|\xi|$ large enough such that $t>10s$.

\textbf{Step 1}: 
Let $e_{\xi}(x)$ be the function $e^{i\xi\phi( x)}r( x)$. Using Proposition \ref{prop:stoptime} and the Cauchy-Schwarz inequality, we have
\begin{align*}
\nonumber\left|\int_X e_\xi(x)\dd\nu(x)\right|&=\left|\int_{g\in M_t^+}\int_Xe_{\xi}(gx)\dd\nu(x)\dd\bar\mu(g)-\int_{g\in M_t^-}\int_Xe_{\xi}(gx)\dd\nu(x)\dd\bar\mu(g)\right|\\
\nonumber&\leq \bar{\mu}(M^+_t)^{1/2}\left(\int_{M^+_t}\left|\int_Xe_{\xi}(gx)\dd\nu(x)\right|^2\dd\bar{\mu}(g)\right)^{1/2}\\
&\quad+\bar{\mu}(M^-_t)^{1/2}\left(\int_{M^-_t}\left|\int_Xe_{\xi}(gx)\dd\nu(x)\right|^2\dd\bar{\mu}(g)\right)^{1/2}.
\end{align*}

By Lemma \ref{lem:ntfin} and Proposition \ref{prop:stoptime}, $\bar{\mu}(M^+_t),\ \bar{\mu}(M^-_t)$ are uniformly bounded with $t$. Change the order of integration, then
\begin{align}\label{ineq:causch}
\nonumber\left|\widehat{\phi_*(r\dd\nu)}(\xi)\right|&\lesssim\left (\int_{X^2}\int_{M_t^+}e^{i\xi(\phi( gx)-\phi( gy))}r( gx)r( gy)\dd\bar{\mu}(g)\dd\nu(x)\dd\nu(y)\right)^{1/2}\\
& \qquad+\left(\int_{X^2}\int_{M_t^-}e^{i\xi(\phi( gx)-\phi( gy))}r( gx)r( gy)\dd\bar{\mu}(g)\dd\nu(x)\dd\nu(y)\right)^{1/2}.
\end{align}
From now on, we only consider $M_t^+$. The set $M_t^-$ has similar properties, and the needed changes will be discussed in remarks, which appear at the end of each section. 

\textbf{Step 2}:
The main approximation, which will be proved in Section \ref{secapproxi}, replaces the distance $ \phi(gx)-\phi(gy)$ with $\phi'e^{-2\kappa(g)}d(x,y)$. The intuition here is that in a large set, whose complement has exponentially small measure, the behavior is nice.

To apply replacement, some regularity conditions on $x,y$ and $g$ are needed. Define a subset of $M_t^+$ for $x,y$ in $X$ by
\begin{equation}
\begin{split}
M_t^+(x,y)=\{g\in M_t^+||\kappa(g)-\kappa(Tg)|<\epsilon_3s,d(x_g^m,g^{-1}x)<e^{-t},
d(g^{-1}x,x),d(g^{-1}x,y)>2e^{-\epsilon_3s} \}.
\end{split}
\end{equation}
For fixed $x,y$, set
\begin{align*}
&\cartwo_0(g)=e^{i\xi(\phi( gx)-\phi( gy))}r( gx)r( gy),\\
&\cartwo_1(g)=e^{i\xi\sgn(g^{-1}x,x,y)\phi'(gx)d(x,y)\exp(-2\kappa(g))/(d(g^{-1}x,x)d(g^{-1}x,y))}r( gx)^2.
\end{align*}
We give a control of the error, which appears in the replacement. 
\begin{prop}\label{prop:mainapp}
Assume that $t>2s$. We have an exponential decay for all $g$ in $M_t^+(x,y)$. That is
\begin{align}\label{ineq:mainapp}
|\cartwo_0(g)-\cartwo_1(g)|=\oexp(s).
\end{align}
\end{prop}
This property will be proved in Section \ref{secapproxi}.
We want to use some smooth cutoffs to regularize the function $\cartwo_1(g,x,y)$. Let $\rho$ be a smooth function on $\bb R$ such that $\rho|_{[-1,1]}=1$, $\rho$ takes values in $[0,1]$, $\supp\rho\subset[-2,2]$ and $|\rho'|\leq 2$. Let 
\begin{equation}
\cartwo_2(g)=\cartwo_1(g)(1-\rho(d(g^{-1}x,x)e^{\epsilon_3s}))(1-\rho(d(g^{-1}x,y)e^{\epsilon_3s}))\rho(\frac{\kappa(g)-\kappa(Tg)}{\epsilon_3s})\rho(\frac{\kappa(Tg)-t}{2\epsilon_3s}).
\end{equation}
When $d(g^{-1}x,x)<e^{-\epsilon_3s}$ or $d(g^{-1}x,y)<e^{-\epsilon_3s}$, the function $\cartwo_2$ will be 0. With fixed $x,y$, $\sgn(g^{-1}x,x,y)$ is a function of $g^{-1}x$, and the discontinuity is at $x$ and $y$. Hence the discontinuity of $\sgn(g^{-1}x,x,y)$ is removed in $\cartwo_2$. 

If $g\in M_t^+(x,y)$, it follows from definition that $|\kappa(Tg)-t|\leq|\kappa(g)-\kappa(Tg)|\leq \epsilon_3s$. Then $\cartwo_2=\cartwo_1$. Since $t>10s$, using Corollary \ref{cor:ntsma}, Lemma \ref{lem:ntfin} and \eqref{ineq:mainapp}, we get
\begin{equation}\label{ineq:mainreg}
\begin{split}
\left|\int_{M_t^+}(\cartwo_0-\cartwo_2)\dd\bar\mu(g)\right|&\leq \bar{\mu}(M_t^+-M_t^+(x,y))+\left|\int_{\mtxy}(\cartwo_0-\cartwo_2)\dd\bar{\mu}(g)\right|\\
&=\bar{\mu}(M_t^+-M_t^+(x,y))+\left|\int_{\mtxy}(\cartwo_0-\cartwo_1)\dd\bar{\mu}(g)\right|=\oexp(s).
\end{split}
\end{equation}

\textbf{Step 3}: 
Introduce the residue process for the Cartan projection. This is inspired by the stopping time. For the stopping time, the existence of the limit distribution of the residual waiting time  was proved in \cite{kesten1974renewal}, but in that paper we do not have a rate of convergence, which is necessary in our method. Here we use the transfer operator to get a uniform rate of convergence. It is difficult to treat the stopping time with transfer operators, because the operator will no longer be continuous. However, the residue process, which will be introduced here, can be routinely analyzed by the transfer operator. What's more, we will get the limit distribution of $gx$ and $g^{-1}y$ simultaneously, which is important to us.

We generalize the inverse action on $\Sigma$, letting $g^{-1}=(g_1,\dots,g_\omega)^{-1}=(g_\omega^{-1},\dots,g_1^{-1})$ for $g$ in $\Sigma$. For a subset $M$ of $\Sigma$, set $\iota(M)=\{g^{-1}|g\in M \}$. Let $\check{\mu}$ be the pushforward of $\mu$ by the inverse action. 
Let $t$ be a positive number. Consider the limit of the following quantity as $t\rightarrow\infty$
\begin{align*}
\sum_{n\geq 0}\int_{\kappa(g)< t\leq \kappa(hg) } f((hg)^{-1}x',hgx,\kappa(hg)-\kappa(g),\kappa(g)-t)\dd\mu(h)\dd\mu^{*n}(g),
\end{align*}
where $x,x'$ are points in $X$ and $f$ is a smooth, compactly supported function on $X^2\times\bb R^2$. Our result is similar to renewal theory. By Proposition \ref{prop:rescar}, when $t$ tends to infinity, the limit is 
\begin{align*}
\int_{X^2}\int_G\int_{-\sigma(h,y)}^{0} f(y',hy,\sigma(h,y),u)\dd u\dd\mu(g)\dd\nu(y)\dd\check\nu(y'),
\end{align*}
where $\check{\nu}$ is the stationary measure of $\check{\mu}$ and the integral $\int_{-\sigma(h,y_1)}^{0}=0$ if $\sigma(h,y_1)<0$.

Since $(Tg)^{-1}=L(g^{-1})$ and $\kappa(g^{-1})=\kappa(g)$, we can define
\begin{equation}\label{equ:ntmt}
N_t^+=\iota(M_t^+)=\{g\in\Sigma|\kappa(Lg)< t\leq \kappa(g) \}.
\end{equation}
Therefore
\[\int_{M_t^+}\cartwo_2(g)\dd\bar{\mu}(g)=\int_{N_t^+}\cartwo_2(g^{-1})\dd\bar{\check{\mu}}(g). \]
Recall that $x,y,\rho$ are fixed. For $x_1,x_2$ in $X$ and $v,u$ in $\bb R$, define
\begin{align*}
&\lambda(x_1,x_2)=d(x,y)e^s\sgn(\xi)\sgn(x,y,x_2)\phi'( x_1)/(d(x_2,x)d(x_2,y)),\\
&\varphi(x_1,x_2,v,u)=r( x_1)^2\times (1-\rho(d(x_2,x)e^{\epsilon_3s}))(1-\rho(d(x_2,y)e^{\epsilon_3s}))\rho(\frac{v}{\epsilon_3s})\rho(\frac{u}{2\epsilon_3s}).
\end{align*} 
By the relation $\xi=\sgn(\xi)e^{2t+s}$, regroup the terms and rewrite the function 
\begin{equation}
	\begin{split}
		\cartwo_2(g^{-1})=e^{i\lambda(g^{-1}x,gx)\exp(-2(\kappa(g)-t)) }\varphi(g^{-1} x,gx,\kappa(g)-\kappa(Lg),\kappa(Lg)-t).
	\end{split}
\end{equation}
Note that the function $\lambda$ is not continuous, but the function $\varphi$ will remove the discontinuity as we have discussed in Step 2. 
In the language of the residue process, let $f$ be the function on $X^2\times\bb R^2$ defined by
\begin{align}
	f(x_1,x_2,v,u)=e^{i\lambda(x_1,x_2)\exp(-2(u+v)) }\varphi(x_1,x_2,v,u).
\end{align}
Thus the function $\cartwo_2(g^{-1})$ can be written as
\[\cartwo_2(g^{-1})=f(g^{-1}x,gx,\kappa(g)-\kappa(Lg),\kappa(Lg)-t).\]
By Proposition \ref{prop:rescar}, for $\delta>0$, $t>2(|K|+\delta)$ (where $K$ is the projection of $\supp f$ onto $\bb R_v$), we have
\begin{equation}\label{ineq:mtlam}
\begin{split}
&\int_{M_t^+}\cartwo_2\dd\bar{\mu}(g)=\int_{N_t^+}f\dd\bar{\check{\mu}}(g)\\
=&\int_{X^2}\int_G\int_{-\sigma(h,x_2)}^{0} f(x_1,hx_2,\sigma(h,x_2),u)\dd u\dd\mu(g)\dd\nu(x_2)\dd\check\nu(x_1)+O_K(\delta +O_\delta/t)| f|_\lf.
\end{split}
\end{equation}
Here $|f|_\ck$ is the Lipschitz norm defined by \[|f|_\ck=|f|_\infty+\sup_{(x_1,x_2,v,u)\neq(x_1',x_2',v',u')}\frac{|f(x_1,x_2,v,u)-f(x_1',x_2',v',u')|}{d(x_1,x_1')+d(x_2,x_2')+|v-v'|+|u-u'|}.\]
\begin{lem}\label{lem:error} There exist constants $\delta_0(s)$ and $t(\delta,s)$ such that if $\delta<\delta_0(s)$ and $t>t(\delta,s)$, then
\begin{equation}
O_K(\delta +O_\delta/t)|f|_\ck\leq e^{-s}.
\end{equation}
\end{lem}
\begin{proof}
By the definition of $\rho$ and $f$, the support of $f$ is in the compact set $X^2\times[-4\epsilon_3s,4\epsilon_3s]^2$. The size of $K$, the projection of $\supp f$ onto $\bb R_v$, is bounded by $8\epsilon_3s$. The definition of $\rho$ implies that $f$ is locally Lipschitz. Together with the fact that $f$ is compactly supported, we conclude that $|f|_\ck$ is controlled by $e^{2s}$ independently of $x,y$. Take $\delta$ small enough according to $s$, then take $t$ large enough according to $\delta$ and $s$. We get the inequality.
\end{proof}
\textbf{Step 4:} For the major term in \eqref{ineq:mtlam}, use the following lemma.
\begin{lem}\label{lem:intlam}
	For $b_1<b_2$ and $\lambda$ nonzero, we have
	\begin{equation}
	|\int_{b_1}^{b_2}e^{i\lambda\exp(-u)}\dd u|\leq \frac{2(e^{b_1}+e^{b_2})}{|\lambda|}.
	\end{equation}
\end{lem}
\begin{proof}
	Integration by parts gives
\begin{align*}
\int_{b_1}^{b_2}e^{i\lambda\exp(-u)}\dd u=\int_{b_1}^{b_2}\frac{\partial_u(e^{i\lambda\exp(-u)})}{-i\lambda e^{-u}}\dd u=\frac{e^{i\lambda\exp(-u)}}{-i\lambda e^{-u}}\Big|_{b_1}^{b_2}+\int_{b_1}^{b_2}e^{i\lambda\exp(-u)}\partial_u\left(\frac{1}{-i\lambda e^{-u}}\right)\dd u.
\end{align*}
This implies that
\[
|\int_{b_1}^{b_2}e^{i\lambda\exp(-u)}\dd u|\leq \left|\frac{e^{i\lambda\exp(-u)}}{\lambda e^{-u}}\Big|_{b_1}^{b_2}\right|+\int_{b_1}^{b_2}\partial_u(\frac{e^u}{|\lambda|})\dd u\leq \frac{2(e^{b_1}+e^{b_2})}{|\lambda|}.
\]
The proof is complete.
\end{proof}
When $d(x,y)>e^{-\epsilon_3s}$, due to the definition of $\rho(\frac{v}{\epsilon_3s})$, the major term only integrates on $h,x_2$ such that $|\sigma(h,x_2)|\leq 2\epsilon_3s$. The inequality $|u|\leq |\sigma(h,x_2)|\leq 2\epsilon_3s$ implies that $\rho(\frac{u}{2\epsilon_3s})=1$. By the hypotheses on $\phi$, when $r(x_1)\neq 0$, we have $|\phi'(x_1)|\geq 1/C_1>0$. Therefore \[|\lambda(x_1,hx_2)|=|d(x,y)e^s\sgn(\xi)\sgn(x,y,hx_2)\phi'( x_1)/(d(hx_2,x)d(hx_2,y))|\geq e^{(1-\epsilon_3s)}/C_1.\]
We use Lemma \ref{lem:intlam} to obtain
\[|\int_{-\sigma(h,x_2)}^{0}f\dd u|\leq|r|^2_\infty\left|\int_{0}^{\sigma(h,x_2)}e^{i\lambda \exp(-2u)}\dd u\right| \leq C_1|r|^2_\infty\frac{1+e^{2\epsilon_3s}}{e^{1-\epsilon_3s}}\leq |r|^2_\infty2e^{(3\epsilon_3-1)s}C_1. \]
Combined with \eqref{ineq:mtlam}, they imply that $\int_{M_t^+}\cartwo_2\dd\bar{\mu}(g)=\oexp(s)$. 
When $d(x,y)\leq e^{-\epsilon_3s}$, the H\"{o}lder regularity of stationary measure \eqref{ineq:regsta} implies that
\begin{align*}
\int_{X\times X}\bm 1_{d(x,y)\leq e^{-\epsilon_3s}}\dd\nu(x)\dd\nu(y)\leq \int_X \nu(B(x,e^{-\epsilon_3s}))\dd\nu(x)=\oexp(s).
\end{align*}

Finally we obtain
\begin{align*}
&\int_{X^2}\int_{M_t^+}\cartwo_2(x,y)\dd\bar{\mu}(g)\dd\nu(x)\dd\nu(y)\\
&\leq \int_{X^2}\bm 1_{d(x,y)>e^{-\epsilon_3s}}\int_{M_t^+}\cartwo_2(x,y)\dd\bar{\mu}(g)\dd\nu(x)\dd\nu(y)\\
&+\int_{X^2}\bm 1_{d(x,y)\leq e^{-\epsilon_3s}}\int_{M_t^+}\cartwo_2(x,y)\dd\bar{\mu}(g)\dd\nu(x)\dd\nu(y)\leq \oexp(s)(1+\bar{\mu}(M_t^+)).
\end{align*}
By Lemma \ref{lem:ntfin}, the measure $\bar{\mu}(M_t^+)$ is uniformly bounded. By using \eqref{ineq:causch} and \eqref{ineq:mainreg}, the proof is complete.
\end{proof}
\begin{rem}[Minus case]
For $M_t^-$, we have another version of Lemma \ref{lem:ntfin}, Corollary \ref{cor:ntsma} and Proposition \ref{prop:rescar}. The integral $|\int_{-\sigma(h,y_1)}^{0}f\dd u|$ is replaced by $|\int_0^{-\sigma(h,y_1)}f\dd u|$.
\end{rem}
\begin{rem}\label{rem:poly}
	When $s$ is large and $\xi$ is of size $e^{Cs}$, all the error terms have polynomial decay except the one from Proposition \ref{prop:rescar}. As we have mentioned in Remark \ref{rem:unispe}, a uniform spectral gap makes Proposition \ref{prop:rescar} effective. Then we will have a polynomial decay.
	
	The uniformity with respect to $\|r\|_{C^1},\|\phi\|_{C^2}$ and $1/\inf_{\supp r}|\phi'|$ is due to the fact that all the terms depend only on these norms and the measure $\mu$. 
\end{rem}
\section{Renewal theory}
\label{secrentheory}

We define a renewal operator $R$ as follows. For a positive bounded Borel function $f$ on $X\times\bb R$, a point $x$ in $X$ and a real number $t$, we set
\begin{align*}
Rf(x,t)=\sum_{n=0}^{+\infty}\int_Gf(gx,\sigma(g,x)-t)\dd\mu^{*n}(g).
\end{align*}
Because of the positivity of $f$, this sum is well defined. In \cite{kesten1974renewal}, Kesten proved a renewal theorem for Markov chains, which is valid in our case \cite{guivarchlepage2015}. But a uniform speed of convergence is needed. We will give a proof using the complex transfer operator, which fulfills our demands. The treatment of the transfer operator will be along the path in \cite{boyer2016renewalrd}. The renewal theorem will give us an equidistribution phenomenon, where the key input is non-arithmeticity.

First we give a proof of renewal theorem for good functions. Then we prove some regularity properties and independence properties for the renewal process. These will imply a version of residue process. Finally, we prove a theorem for the Cartan projection from a similar theorem for the cocycle.

Fix the constant $\epsilon=\sigma_\mu/4$ in this section. Keep in mind that the assumptions of Theorem \ref{thm:fourier} are always satisfied.
\subsection{Complex transfer operators}
We introduce the complex transfer operator $P(z)$. Let $\cal H^{\gamma}(X)$ be the space of $\gamma$-H\"older functions on $X$, a Banach space with the norm $|f|_{\gamma}=|f|_{\infty}+m_\gamma(f)=|f|_\infty+\sup_{x\neq y}\frac{|f(x)-f(y)|}{d(x,y)^\gamma}$. For $f$ in $ \cal H^{\gamma}(X)$ and a complex number $z$, define
\begin{align*}
	P(z)f(x)=\int_{G}e^{-z\sigma(g,x)}f(gx)\dd\mu(g).
\end{align*} 
The main properties of $P(z)$ are summarized as follows
\begin{prop}\label{prop:invtran}\cite[Theorem 4.1, Lemma 4.7]{boyer2016renewalrd}
	For any $\gamma>0$ small enough, there exists $\eta>0$ such that when $|\Re z|<\eta$, the transfer operator $P(z)$ is a bounded operator on $\cal H^{\gamma}(X)$ and depends analytically on $z$. Moreover there exists an analytic operator $U(z)$ on a neighborhood of $0\leq\Re z< \eta$ such that the following equality holds for $0\leq \Re z<\eta$
	\begin{equation}\label{equ:i-pz}
	(I-P(z))^{-1}=\frac{1}{\sigma_{\mu}z}N_0+U(z),
	\end{equation}
	where $N_0$ is the operator defined by $N_0f=\int f\dd\nu$
\end{prop}
\begin{rem}
	In Proposition \ref{prop:invtran}, the non-arithmeticity is crucial to prove that $(I-P(z))^{-1}$ has only one pole in the imaginary axis, which is $0$. The non-arithmeticity follows from Zariski density. See for instance \cite{benoist_proprietes_2000} and \cite{dalbo_topologie_2000}.
\end{rem}
The assumption of Theorem 4.1 in \cite{boyer2016renewalrd} are complicated. It is verified, in the proof of theorem 1.4, page 8 \cite{boyer2016renewalrd}, that our condition on $\mu$ is enough to apply Theorem 4.1. The idea is due to Guivarch and Le Page.
\begin{prop}\cite[Lemma 4.4]{boyer2016renewalrd}\label{prop:priorestimate}
For any $\gamma>0$ small enough, there exist $\eta>0$, $0<\rho<1$, $C>0$ such that when $0\leq\Re z<\eta$, for a natural number $n$ and a $\gamma$-H\"older function $f$, we have
\begin{align}
\label{ineq:pznfin}&|P(z)^nf|_{\infty}\leq (C\rho^{n})^{\Re z}|f|_{\infty}
\end{align}
\end{prop}
\begin{rem}\label{rem:gameps}
	For further usage, we need a bound on $\gamma$. Let $\epsilon,\epsilon'(\epsilon)$ be the two constants in \eqref{ineq:regcon}, that is $\mu^{*n}\{d(gx,x')\leq e^{-\epsilon s} \}\leq Ce^{-\epsilon' s}$, and $\epsilon_1$ the constant in exponential moment. Choose a small $\gamma$ such that $\gamma\leq \frac{1}{4}\max\{\frac{\sigma_\mu/4}{\epsilon'(\sigma_\mu/4)},\epsilon_1 \}$.
\end{rem}
\subsection{Renewal theory for regular functions}
We start to compute the renewal operator. A result for the renewal operator for ``good" functions will be proved. Let $f$ be a function on $X\times \bb R$. Define a norm by
$|f|_{L^\infty\cal H^\gamma}=\sup_{\xi\in\bb R}|f(x,\xi)|_{\cal H^\gamma}$, which is the supremum of the H\"older norm of $f(\cdot,\xi)$. Define another norm $|f|_{W^{1,\infty}\cal H^{\gamma}}=|f|_{L^\infty\cal H^\gamma}+|\partial_\xi f|_{L^\infty\cal H^\gamma}$. Write the Fourier transform $\hat f(x,\xi)=\int e^{i\xi u}f(x,u)\dd u$.
\begin{prop}\label{prop:renreg}
	Let $f$ be a positive bounded continuous function in $L^1(X\times\bb R,\nu\otimes Leb )$ such that its Fourier transform satisfies $\hat f\in L^{\infty}\cal H^\gamma$ and $\partial_\xi \hat f\in L^\infty\cal H^\gamma$. Assume that the projection of $\supp \hat{f}$ onto $\bb R$ is in a compact set $K$. Then for all $t>0$ and $x$ in $X$, we have
	\[Rf(x,t)=\frac{1}{\sigma_{\mu}}\int_{-t}^{\infty}f(y,u)\dd u\dd\nu(y)+ \frac{1}{t}O_K(|\hat f|_{W^{1,\infty}\cal H^{\gamma}}). \]
\end{prop}
\begin{proof}
	Combine the following two lemmas.
\end{proof}
\begin{lem}\label{lem:rencon}
Under the same assumption as in Proposition \ref{prop:renreg}, we have
\begin{align*}
Rf(x,t)=\frac{1}{\sigma_{\mu}}\int_{t}^{\infty}f(y,u)\dd u\dd\nu(y)+\frac{1}{2\pi}\int e^{it\xi}U(i\xi)\hat{f}(x,\xi)\dd\xi.
\end{align*}
\end{lem}
\begin{proof}
Introduce a local notation: for $(x,t)$ in $X\times\bb R$ and $s\geq 0$, write
\begin{align*}
\pexp f(x,t)=\int_{G}e^{-s\sigma(g,x)}f(gx,\sigma(g,x)+t)\dd\mu(g).
\end{align*}
When $s=0$, we abbreviate the notation $\pexpo_0$ to $\pexpo$. We want to prove the following equality,
\begin{equation}\label{equ:slim}
\sum_{n\geq 0} \pexpo^n(f)(x,t)=\lim_{s\rightarrow 0^+}\sum_{n\geq 0} \pexp^n(f)(x,t).
\end{equation}
By definition, one has
\begin{align*}
\pexp^n(f)(x,t)&=\int_G e^{-s\sigma(g,x)}f(gx,\sigma(g,x)+t)\dd\mu^{*n}(g)\\
&=\int_G e^{-s\sigma(g,x)}(\bm 1_{\sigma(g,x)>0}+\bm 1_{\sigma(g,x)\leq 0})f(gx,\sigma(g,x)+t)\dd\mu^{*n}(g).
\end{align*}
\begin{itemize}
\item The part $\bm 1_{\sigma(g,x)>0}$, since $f\geq 0$, use the monotone convergence theorem. When $s\rightarrow 0^+$ then
\begin{align*}
\sum_{n\geq 0}\int_G e^{-s\sigma(g,x)}\bm 1_{\sigma(g,x)>0}f(gx,\sigma(g,x)+t)\dd\mu^{*n}(g)\rightarrow\sum_{n\geq 0} \int_G\bm 1_{\sigma(g,x)>0}f(gx,\sigma(g,x)+t)\dd\mu^{*n}(g).
\end{align*}
\item For the part $\bm 1_{\sigma(g,x)\leq 0}$, take $s$ in $[0,\eta/2]$. Proposition \ref{prop:priorestimate} implies that
\[\int_G e^{-s\sigma(g,x)}\bm 1_{\sigma(g,x)\leq 0}f(gx,\sigma(g,x)+t)\dd\mu^{*n}(g)
\leq \int_G e^{-\eta\sigma(g,x)/2}|f|_\infty\dd\mu^{*n}(g)
\leq (C\rho^n)^{\eta/2}|f|_{\infty}.\] 
Since $\sum_{n\geq 0}\rho^{n\eta/2}$ is finite, take $e^{-\eta\sigma(g,x)/2}|f|_\infty$ as the dominant function. Then use the dominated convergence theorem to conclude. 
\end{itemize}
This proves equation \eqref{equ:slim}.

Using the inverse Fourier transform, we have
\begin{equation}\label{equ:bsn}
\begin{split}
\sum_{n\geq 0} \pexp^n(f)(x,t)&=\sum_{n\geq 0}\int_Ge^{-s\sigma(g,x)}f(gx,\sigma(g,x)+t)\dd\mu^{*n}(g)\\
&=\sum_{n\geq 0}\int_Ge^{-s\sigma(g,x)}\frac{1}{2\pi}\int_{\bb R}e^{i\xi(\sigma(g,x)+t)}\hat{f}(gx,\xi)\dd\xi\dd\mu^{*n}(g).
\end{split}
\end{equation}
Since $\hat{f}(x,\xi)$ has compact support, $|\hat{f}(x,\xi)|\leq |\hat f(x,\xi)|_{L^\infty_\xi\cal H^\gamma}$ and $|P(s)^n\bm 1|\leq C\rho^{sn}$ for $s$ in $[0,\eta/2]$ (Proposition \ref{prop:priorestimate}), we have
\begin{align*}
\sum_{n\geq 0}\int_Ge^{-s\sigma(g,x)}\int_{\bb R}|\hat{f}(gx,\xi)|\dd\xi\dd\mu^{*n}(g)\leq C_{f}\sum_{n\geq 0}\int_Ge^{-s\sigma(g,x)}\dd\mu^{*n}(g)=C_{f}\sum_{n\geq 0} P(s)^n(\bm 1)< \infty,
\end{align*} 
which implies that the right hand side of \eqref{equ:bsn} is absolutely convergent.
Consequently, we can use the Fubini theorem to change the order of the integration. By the hypothesis $\hat{f}(x,\xi)\in\cal H^{\gamma}(X)$, Proposition \ref{prop:invtran} implies that
\begin{align*}
\sum_{n\geq 0} \pexp^n(f)(x,t)&=\frac{1}{2\pi}\int_{\bb R}\sum_{n\geq 0}\int_G e^{(-s+i\xi)\sigma(g,x)}\hat{f}(gx,\xi)\dd\mu^{*n}(g)e^{it\xi}\dd\xi\\
&=\frac{1}{2\pi}\int_{\bb R}\sum_{n\geq 0} P^n(s-i\xi)\hat{f}(x,\xi)e^{it\xi}\dd\xi\\
&=\frac{1}{2\pi}\int_{\bb R}(1- P(s-i\xi))^{-1}\hat{f}(x,\xi)e^{it\xi}\dd\xi\\
&=\frac{1}{2\pi}\int_{\bb R}(\frac{N_0}{\sigma_{\mu}(s-i\xi)}+U(s-i\xi))\hat{f}(x,\xi) e^{it\xi}\dd\xi.
\end{align*}
Since $\frac{1}{s-i\xi}=\int_{0}^{+\infty}e^{-(s-i\xi)u}\dd u$ for $s>0$, together with the property $\hat{f}\in L^1(\bb R)$, we have
\begin{align*}
\frac{1}{2\pi}\int_{\bb R}\frac{N_0}{\sigma_{\mu}(s-i\xi)}\hat{f}(x,\xi) e^{it\xi}\dd\xi&=\frac{1}{2\pi}\frac{1}{\sigma_{\mu}}\int_X\int_{\bb R}\frac{\hat{f}(y,\xi)}{s-i\xi}e^{it\xi}\dd\xi\dd\nu(y)\\
&=\frac{1}{\sigma_{\mu}}\int_X\int_{0}^{\infty}f(y,u+t)e^{-su}\dd u\dd\nu(y).
\end{align*}
When $s\rightarrow 0^+$, since $f$ is integrable with respect to the product measure $\nu\otimes Leb$, by monotone convergence theorem, the limit is $\frac{1}{\sigma_{\mu}}\int_X\int_{t}^{\infty}f(y,u)\dd u\dd\nu(y)$.
Since $\hat f(x,\xi)$ is compactly supported, we have
\begin{align*}
\lim_{s\rightarrow 0^+}\int_{\bb R}U(-s+i\xi)\hat{f}(x,\xi) e^{it\xi}\dd\xi=\int_{\bb R}U(i\xi)\hat{f}(x,\xi) e^{it\xi}\dd\xi.
\end{align*}
The proof is complete.
\end{proof}

\begin{lem}\label{lem:renres} 
Under the same assumption as in Proposition \ref{prop:renreg}, we have
\[|\int e^{-it\xi}U(i\xi)\hat{f}(x,\xi)\dd\xi|\leq \frac{1}{t}O_K\left(|\hat f|_{L^\infty\cal H^{\gamma}}+|\partial_\xi\hat f|_{L^\infty\cal H^{\gamma}}\right). \]
\end{lem}
\begin{proof}
	 Use the fact that $\hat f(x,\xi)$ is compactly supported and $|\hat{f}(x,\xi)|_{\cal H^\gamma},\,|\partial_{\xi}\hat{f}(x,\xi)|_{\cal H^\gamma}<\infty$. Then applying integration by parts, we have
	\begin{align*}
	\int e^{-it\xi}U(i\xi)\hat{f}(x,\xi)\dd\xi&=\frac{1}{it}\int e^{-it\xi}\partial_{\xi}(U(i\xi)\hat{f}(x,\xi))\dd\xi\\
	&=\frac{1}{it}\int e^{-it\xi}\left(\partial_{\xi}(U(i\xi))\hat{f}(x,\xi)+U(i\xi)\partial_{\xi}\hat{f}(x,\xi)\right)\dd\xi.
	\end{align*}
	Since the operator norms of $U(i\xi)$ and $\partial_\xi U(i\xi)$ are uniformly bounded on compact regions, the result follows.
\end{proof}

\subsection{Regularity properties of renewal measures}\label{sec:regular}
We have two principles in this subsection. \textbf{Principle 1}: Let $f$ be a bounded Borel function supported in $X\times[0,a]$. When we take the renewal sum outside of the interval $\ntint=[\frac{t}{\sigma_\mu+\epsilon},\frac{t+a}{\sigma_\mu-\epsilon}]$,
\[\sum_{n\in\bb N-I_t}\int_Gf(gx,\sigma(g,x)-t)\dd\mu^{*n}(g)=\sum_{n\in\bb N-I_t}\int_Gf(gx,\sigma(g,x)-t)\bm 1_{[0,a]}(\sigma(g,x)-t)\dd\mu^{*n}(g),\]
this sum decays exponentially with $t$. This is given by the large deviations principle (Corollary \ref{cor:lardev}, \ref{cor:lardev1}). For $n$ in the interval $\ntint$, if some property is valid for each $n$ with an exponential error of $n$, we sum up. Since the length of this interval is comparable with $t$, this property is also valid for the renewal sum with an exponential error of $t$.

\textbf{Principle 2}: The other is independence. By Proposition \ref{prop:renreg}, the limit distribution of $(\sigma(g,x)-t,gx)$ is $\frac{1}{\sigma_\mu}\nu\otimes Leb$, which is a product measure. That roughly means the following: As in Remark \ref{rem:ransto}, let $X_n=b_n\cdots b_1$ be a random walk on $G$. Let $F=F_1\times F_2$ where $F_1$, $F_2$ are Borel subsets of $X$, $\bb R$ respectively. Then 
\[\sum_{n\geq 0}\bb P\{(X_nx,\sigma(X_n,x)-t)\in F_1\times F_2 \}\rightarrow \frac{1}{\sigma_\mu}\nu(F_1)\otimes Leb(F_2)\text{ as }t\rightarrow+\infty. \] 
More concretely, we could expect that $R\bb (\bm 1_{F_1\times F_2})(x,t)$ is almost $\frac{1}{\sigma_\mu}\nu(F_1)\otimes Leb(F_2)$ when $t$ is large.

We want to use convolution to smooth out the target function. There exists an even function $\psi$ such that it is a probability density, and the Fourier transform $\hat{\psi}$ is compactly supported. Let $\psi_{\delta}(t)=\frac{1}{\delta^2}\psi(\frac{t}{\delta^2})$. Then $\int_{-\delta}^{\delta}\psi_{\delta}(t)\dd t=\int_{-1/\delta}^{1/\delta}\psi(t)\dd t>1-C\delta$.
\begin{prop}\label{prop:renint}
Let $\delta\leq 1/3$ and $b_2\geq b_1$. If $b_2-b_1\geq 2\delta$, then for $x$ in $X$ and $t>0$, we have
\begin{equation}\label{ineq:renint}
R(\bm 1_{[b_1,b_2]})(x,t)\lesssim (b_2-b_1)(1/\sigma_{\mu}+C_{\delta}(1+|b_2|+|b_1|)/t).
\end{equation}
If $0\leq b_2-b_1<2\delta$, then for $x$ in $X$ and $t>0$, we have
\begin{equation}\label{ineq:renintsma}
R(\bm 1_{[b_1,b_2]})(x,t)\lesssim \delta(1/\sigma_{\mu}+C_{\delta}(1+|b_1|)/t).
\end{equation}
\end{prop}
\begin{proof}
When $b_2-b_1\geq 2\delta$, if $u$ is in $[b_1,b_2]$, then $[u-b_2,u-b_1]$ contains at least one of $[0,\delta]$ or $[-\delta,0]$. Therefore
\[\psi_{\delta}*\bm 1_{[b_1,b_2]}(u)=\int_{b_1}^{b_2}\psi_{\delta}(u-v)\dd v\geq \int_{0}^{\delta}\psi(v)\dd v\geq (1-\delta)/2. \]
Then
\begin{equation}\label{ineq:psidellarg}
\bm 1_{[b_1,b_2]}\leq 3\psi_{\delta}*\bm 1_{[b_1,b_2]}.
\end{equation}
It is sufficient to bound $R(\psi_{\delta}*\bm 1_{[b_1,b_2]})$. Proposition \ref{prop:renreg} implies that
\[R(\psi_{\delta}*\bm 1_{[b_1,b_2]})=\frac{1}{\sigma_{\mu}}\int_{-t}^{\infty}\psi_{\delta}*\bm 1_{[b_1,b_2]}+\frac{O_\delta}{t}|\hat{\psi_\delta}\hat{\bm 1}_{[b_1,b_2]}|_{W^{1,\infty}\cal H^\gamma}. \]
The first term is less than $\int\psi_{\delta}*\bm 1_{[b_1,b_2]}=(b_2-b_1)$. For the second term, we have
\begin{align*}
|\hat{\psi_\delta}\hat{\bm 1}_{[b_1,b_2]}|_{W^{1,\infty}\cal H^\gamma}&=|\hat{\psi_\delta}\hat{\bm 1}_{[b_1,b_2]}|_{L^\infty\cal H^\gamma}+|\partial_\xi\hat{\psi_\delta}\hat{\bm 1}_{[b_1,b_2]}|_{L^{\infty}\cal H^\gamma}\\
&=|\hat{\psi_\delta}\hat{\bm 1}_{[b_1,b_2]}|_{L^\infty}+|\partial_\xi\hat{\psi_\delta}\hat{\bm 1}_{[b_1,b_2]}|_{L^{\infty}} \\
&\leq C_\delta'(|\bm 1_{[b_1,b_2]}(u)|_{L^1}+|u\bm 1_{[b_1,b_2]}(u)|_{L^1})\leq C_\delta'(b_2-b_1)(1+|b_1|+|b_2|).
\end{align*}

When $b_2-b_1\in [0,2\delta]$, the renewal sum $R(\bm 1_{[b_1,b_2]})$ is bounded by $R(\bm 1_{[b_1,b_1+2\delta]})$. Then use the previous case.
\end{proof}

In Proposition \ref{prop:renreg}, since we do not have a good control of the spectral radius of the operator $U(i\xi)$ for large $|\xi|$, the estimates are effective only for large $t$, which means that when $t$ is small the error term will be out of control. The following lemma combines the transfer operator and the large deviations principle to give a uniform estimate.
\begin{lem}\label{lem:renintts}
For real numbers $s,t$ and  a point $x$ in $X$, we have
\begin{equation}\label{ineq:renintts}
R(\bm 1_{[0,s]})(x,t)\lesssim \max\{1,s \}.
\end{equation}
\end{lem}
\begin{proof}
	We can suppose that $s>1$. If not, then $R(\bm 1_{[0,s]})(x,t)\leq R(\bm 1_{[0,1]})(x,t)$. When $t\geq s$, this is a direct corollary of Proposition \ref{prop:renint}. Fixing $\delta=1/3$, we get
	\[\frac{1+|b_1|+|b_2|}{t}\leq\frac{1+s}{t}\leq 2. \]
	Then $R(\bm 1_{[0,s]})(x,t)\lesssim s(1/\sigma_\mu+2C_\delta)$. 
	
	When $t<s$, let $m=[\max\{0,(t+s)/(\sigma_\mu-\epsilon) \}]+1$. By Corollary \ref{cor:lardev}, we have
	\begin{align*}
	 R(\bm 1_{[0,s]})(x,t)\leq R(\bm 1_{[0,2s]})(x,0)
	&\leq \sum_{n\leq m}\mu^{\otimes n} \{\sigma(g,x)\leq 2s \}+\sum_{n>m}\mu^{\otimes n}\{\sigma(g,x)\leq 2s \}\\
	&\lesssim m+e ^{-\epsilon'm}\lesssim s.
	\end{align*}
	The proof is complete.
\end{proof}

In the renewal theorem, the limits of the scalar part $\sigma(g,x)$ and the angle part $gx$ are independent. Using this spirit, we give the following lemma, which quantifies this independence. In the proof, when $t$ is large enough, using Proposition \ref{prop:renreg}, the remainder term will be small. When $t$ is small, we have another estimate from the regularity of the convolution measure $\mu^{\otimes n}$.
\begin{prop}\label{prop:cocind}
For $s>0,\,a>0,\ t>5s$, and $x,\,x'$ in $X$, we have
\begin{align}\label{ineq:cocind}
R(\bm 1_{B(x',e^{-s})\times[0,a]})(x,t)=(1+a)^2\oexp(s).
\end{align}
\end{prop}
\begin{proof} Decompose the region of $t$ into two parts:
\begin{itemize}
	\item When $5s<t\leq e^{2\gamma s}$, by Corollaries \ref{cor:lardev}, \ref{cor:lardev1}, it suffices to consider $n\in[t/(\sigma_{\mu}+\epsilon),(t+a)/(\sigma_{\mu}-\epsilon)]$. Due to the hypothesis in this situation $s\leq t/5=\epsilon t/(\sigma_\mu+\epsilon)\leq \epsilon n$, we can use Corollary \ref{cor:regcon} to obtain
	\[\mu^{*n}\{d(gx,x')\leq e^{-s} \}\lesssim e^{-\epsilon' s/\epsilon}. \]
	Then the measure of this part, summing up the above inequality over all $n\in[t/(\sigma_{\mu}+\epsilon),(t+a)/(\sigma_{\mu}-\epsilon)]$, is less than $C(t+a)e^{-\epsilon's/\epsilon}\lesssim (1+a)e^{-\gamma s}$ (here we use the Remark \ref{rem:gameps}, $4\gamma\leq \epsilon/\epsilon'$). 
	\item When $t\geq e^{2\gamma s}$, we take $f=\bm 1_{[0,a]}\varpi(x)$ where $\varpi(x)$ is a function on $X$ such that $\varpi_{B(x',e^{-s})}=1$, $\supp \varpi\subset B(x',2e^{-s})$ and $|\varpi|_\gamma\leq e^{\gamma s}$. As in the proof of Proposition \ref{prop:renint}, we use $\psi_\delta$ to regularize this function. By \eqref{ineq:psidellarg}, we have 
	\[3R(\psi_\delta*f)(x,t)\geq R(\bm 1_{B(x',e^{-s})\times[0,a]})(x,t).\]
	Proposition \ref{prop:renreg} implies
	\[R(\psi_{\delta}*f)=\frac{1}{\sigma_{\mu}}\int_X\int_{-t}^{\infty}\psi_{\delta}*f(x,u)\dd u\dd\nu(x)+\frac{C_\delta}{t}(|\widehat{\psi_\delta *f}|_{W^{1,\infty}\cal H^\gamma}). \]
    Since $\widehat{\psi_\delta*f}(x,\xi)=\hat{\psi}_\delta(\xi)\hat{\bm 1}_{[0,a]}(\xi)\varpi(x)$, the two functions are independent. We can use the same estimate as in the proof of Proposition \ref{prop:renint}. So the rest term is less than $C_\delta'(1+a)^2e^{\gamma s}/t$. The major term, due to the regularity of the stationary measure \eqref{ineq:regsta}, is controlled by $ae^{-\alpha s}/\sigma_\mu$. The result follows from the hypothesis $t>e^{2\gamma s}$.
\end{itemize}
The proof is complete.
\end{proof}
We also need the independence of $\sigma(g,x)$ and $g^{-1}x_o$, where $x,x_o$ are two points in $X$. For proving this property, we pass through the Cartan projection, because the order of products in the Cartan projection can be reversed.
The following proof uses Lemma \ref{lem:coccar}, which is a central tool to prove a renewal type theorem for the Cartan projection from a renewal type theorem for the cocycle.

Let $f$ be a positive bounded Borel function on $X\times \bb R$. For $(x,t)\in X\times\bb R$, we define
\[R_P(f)(x,t)=\sum_{n\geq 0}\int_G f(gx,\kappa(g)-t)\dd\mu^{*n}(g).  \]
\begin{lem}\label{lem:carind}
	For  $s>0,\,a>0,\ t>10s$, and $x,\,x'$ in $X$, we have
	\begin{equation}
	R_P(\bm 1_{B(x',e^{-s})\times[0,a]})(x,t)=(1+a)^2\oexp(s).
	\end{equation}
\end{lem}
\begin{proof}
Due to Corollary \ref{cor:lardev} and Corollary \ref{cor:lardev1}, the sum of the integral of $n\leq t/(\sigma_\mu+\epsilon)$ and $n\geq (t+a)/(\sigma_\mu-\epsilon)$ is exponentially small.

If suffices to consider $n$ in the interval $\ntint=[t/(\sigma_\mu+\epsilon),(t+a)/(\sigma_\mu-\epsilon)]$. Fix $l=[\epsilon_4 t/\sigma_\mu]$ with $\epsilon_4=1/10$. By Lemma 2.9, there exists $S_{n,l,x}\subset G^{\times n}$ such that $\mu^{\otimes n}S_{n,l,x}^c=\oexp(l)$, and for $(g_n,\dots,g_1)$ in $S_{n,l,x}$, letting $g=(g_n,\dots, g_{l+1})$ and $\glg=(g_l,\dots, g_1)$, we have
\begin{align*}
	|\kappa(g\glg)-\sigma(g,\glg x)-\kappa(\glg )|\leq e^{-\epsilon l}\leq 1.
\end{align*}
Thus 
\begin{align*}
&\mu^{\otimes n}\{\kappa(g\glg )\in[t,t+a],d(g\glg x,x')\leq e^{-s} \}\\
\leq& \mu^{\otimes n}\{S_{n,l,x}^c \}+\mu^{\otimes n}\{g\glg \in S_{n,l,x}|\kappa(g\glg )\in[t,t+a],d(g\glg x,x')\leq e^{-s} \} \\
\leq& \oexp(l)+\mu^{\otimes n}\{\sigma(g,\glg x)+\kappa(\glg )\in[t-1,t+a+1],d(g\glg x,x')\leq e^{-s} \}. 
\end{align*}
Therefore summing over $n$ and integrating first with respect to $g$, we get
\begin{equation}\label{ineq:carcocrenind}
\begin{split}
	&\sum_{n\in \ntint}\mu^{\otimes n}\{\kappa(g\glg )\in[t,t+a],d(g\glg x,x')\leq e^{-s} \}\\
	\leq &\sum_{n\in \ntint}\mu^{\otimes n}\{\sigma(g,\glg x)+\kappa(\glg )\in[t-1,t+a+1],d(g\glg x,x')\leq e^{-s} \}+t\oexp(l)
	\\
	\leq& t\oexp(l)+\int R(\bm 1_{B(x',e^{-s})\times[-1,a+1]})(\glg x,t-\kappa(\glg))\dd\mu^{*l}(\glg ).
\end{split}
\end{equation}

Hence, it is sufficient to bound $\int R(\bm 1_{B(x',e^{-s})\times[-1,a+1]})(\glg x,t-\kappa(\glg))\dd\mu^{*l}(\glg )$. Let $G_{l,\epsilon}=\{\glg\in G^{\times l}|\kappa(\glg)\leq l(\sigma_\mu+\epsilon) \}$. By the large deviations principle (Corollary \ref{cor:lardev1}), we have $\mu^{*l}G_{l,\epsilon}^c=\oexp(l)$. 
\begin{itemize}
\item For $\glg \in G_{l,\epsilon}$, we have $t-\kappa(\glg )\geq t-l(\sigma_\mu+\epsilon)=t-\epsilon_4(\sigma_\mu+\epsilon)t/\sigma_\mu>t/2\geq 5s$. Hence, Proposition \ref{prop:cocind} implies that
\begin{equation*}
R(\bm 1_{B(x',e^{-s})\times[-1,a+1]})(\glg x,t-\kappa(\glg))\lesssim (1+a)^2\oexp(s).
\end{equation*}
\item For $\glg \in G_{l,\epsilon}^c$, Lemma \ref{lem:renintts} implies that
\begin{equation*}
R(\bm 1_{B(x',e^{-s})\times[-1,a+1]})(\glg x,t-\kappa(\glg))\lesssim (1+a).
\end{equation*}
\end{itemize}
Combining the above two inequalities, we have
\begin{align}\label{ineq:intren}
\int R(\bm 1_{B(x',e^{-s})\times[-1,a+1]})(\glg x,t-\kappa(\glg))\dd\mu^{*l}(\glg )\lesssim (1+a)^2\oexp(s)+\oexp(l)(1+a).
\end{align}
The proof is complete.
\end{proof}
There is a byproduct of the above lemma. When the function $f$ does not depend on $X$, abbreviate $R_P(f)(x,t)$ by $R_P(f)(t)$.
\begin{lem}\label{lem:carint}
For real numbers $s,t$,  we have
\begin{equation}
\label{ineq:catint}R_P(\bm 1_{[0,s]})(t)\lesssim \max\{1,s^2 \}.
\end{equation}
\end{lem}
\begin{rem}
Here the term $s^2$ is not optimal. With some extra work, it can be improved to $s$.
\end{rem}
\begin{proof}
Suppose that $s\geq1$. If not, then $R_p(\bm 1_{[0,s]})(x,t)\leq R_p(\bm 1_{[0,1]})(x,t)$. 
When $t\geq 10$, apply Lemma \ref{lem:carind} with $a=s, e^{-s}=e^{-1}, x'=x_j, j\in J$, where $J$ is a finite set such that $\cup_{j\in J}B(x_j,e^{-1})$ covers $X$. So we get $R_P(\bm 1_{[0,s]})(t)\lesssim s^2$.

When $t<10\leq 10s$, let $m=[\max\{0,(t+s)/(\sigma_\mu-\epsilon) \}]+1$. By Corollary \ref{cor:lardev1}, we have
\begin{align*}
 R_P(\bm 1_{[0,s]})(t)\leq R_P(\bm 1_{[0,2s]})(0)
&\leq\sum_{n\leq m}\mu^{* n} \{\kappa(g)\leq t+s \}+\sum_{n>m}\mu^{* n}\{\kappa(g)\leq t+s \} \\
&\lesssim m+e ^{-\epsilon'm}\lesssim s.
\end{align*}
The proof is complete.
\end{proof}

Now we are going to prove the independence of $\sigma(g,x)$ and $g^{-1}x$. Recall that $\check{\mu}$ is the pushforward of $\mu$ by the inverse action. Let $f$ be a positive bounded Borel function on $X\times \bb R$. For $(x_o,x,t)\in X^2\times\bb R$, we define
\[R_I(f)(x_o,x,t)=\sum_{n\geq 0}\int_G f(g^{-1}x_o,\sigma(g,x)-t)\dd\mu^{*n}(g).  \]
\begin{prop}\label{prop:cocindinv}
For  $s>0,\,a>0,\ t>\max\{10s,10 \}$, and $x,\,x',\,x_o$ in $X$, we have
	\begin{equation}\label{ineq:cocindinv}
		R_I(\bm 1_{B(x',e^{-s})\times[0,a]})(x_o,x,t)=(1+a)^2\oexp(s).
	\end{equation}
\end{prop}
\begin{proof}
Due to Corollary \ref{cor:lardev} and Corollary \ref{cor:lardev1}, the sums of the integral of $n\leq t/(\sigma_\mu+\epsilon)$ and $n\geq (t+a)/(\sigma_\mu-\epsilon)$ is exponentially small. 

It suffices to consider $n$ in the interval $\ntint=[t/(\sigma_\mu+\epsilon),(t+a)/(\sigma_\mu-\epsilon)]$. Let 
$$G_{\epsilon,n}=\{g\in G^{\times n}|\kappa(g)\geq n(\sigma_\mu-\epsilon/2),d(g^{-1}x_o,x)>e^{-\epsilon n},d(g^{-1}x_o,x_g^m)\leq e^{-(2\sigma_\mu-\epsilon)n} \}.$$ 
By inequalities \eqref{ineq:lardev1}, \eqref{ineq:regcon} and \eqref{ineq:xgmgx}, we have $\mu^{\otimes n}G_{\epsilon,n}\geq 1-\oexp(n)$. Since $t>10$, for $n$ in $\ntint$, we have $n\geq t/(\sigma_\mu+\epsilon)\geq 10/(\sigma_\mu+\epsilon)$. For $g\in G_{\epsilon,n}$, we have
\[\frac{e^{-2\kappa(g)}+d(x_g^m,g^{-1}x_o)}{d(g^{-1}x_o,x)}\leq \frac{2e^{-(2\sigma_\mu-\epsilon)n}}{e^{-\epsilon n}}=2e^{-(2\sigma_\mu-2\epsilon)n}\leq 2e^{-20(\sigma_\mu-\epsilon)/(\sigma_\mu+\epsilon)}\leq 1/2. \]
Using Lemma \ref{lem:carcoc} with $g\in G_{\epsilon,n}$, we have
\[|\sigma(g,x)-\kappa(g)-\log d(g^{-1}x_o,x)|\leq 2\frac{e^{-2\kappa(g)}+d(x_g^m,g^{-1}x_o)}{d(g^{-1}x_o,x)}\leq 4e^{-(2\sigma_\mu-2\epsilon)n}\leq 1. \]
Therefore,
\begin{equation*}
\begin{split}
&\mu^{\otimes n}\{\sigma(g,x)\in[t,t+a],\ d(g^{-1}x_o,x')\leq e^{-s} \}\leq \oexp(n)+\\
&\mu^{\otimes n}\{\kappa(g)\in[t-1,t+a+1]-\log d(g^{-1}x_o,x),d(g^{-1}x_o,x')\leq e^{-s} \}.
\end{split}
\end{equation*}
Summing up over $\ntint$ and using the definition of $\check{\mu}$, we have
\begin{equation}\label{ineq:cocindinvcar}
\begin{split}
&\sum_{n\in\ntint}\mu^{\otimes n}\{\sigma(g,x)\in[t,t+a],\ d(g^{-1}x_o,x')\leq e^{-s} \}\\
\leq& \oexp(t)+\sum_{n\geq 0}\mu^{\otimes n}\{\kappa(g)\in[t-1,t+a+1]-\log d(g^{-1}x_o,x),d(g^{-1}x_o,x')\leq e^{-s} \}\\
=& \oexp(t)+\sum_{n\geq 0}\check\mu^{\otimes n}\{\kappa(g)\in[t-1,t+a+1]-\log d(gx_o,x),d(gx_o,x')\leq e^{-s} \}.
\end{split}
\end{equation}

Hence, it is sufficient to bound $R_P(\bm 1_{u+\log d(y,x)\in[-1,a+1],d(y,x')\leq e^{-s}})(x_o,t)$, where we use $(y,u)$ to denote the variables, and the measure $\mu$ is replaced by $\check{\mu}$. For simplicity, we use the same notation $R_P$. Cutting the region along $\{y\in X|\log d(y,x)\leq -t_1\}$ and the subsets $\{y\in X|\log d(y,x)\in[-(k+1)s,-ks]\}$ for $0\leq k< t_1/s$, where $t_1=(t-1)/9$. 
\begin{itemize}
\item When $k=0$, since $t-1>10s$, we can use Lemma \ref{lem:carind} to obtain
\begin{align*}
& R_P(\bm 1_{u+\log d(y,x)\in[-1,a+1],d(y,x')\leq e^{-s},d(y,x)\geq e^{-s}})(x_o,t)\\
\leq & R_P(\bm 1_{d(y,x')\leq e^{-s},u\in[-1,s+a+1]})(x_o,t)\lesssim (1+s+a)^2e^{-\epsilon's}.
\end{align*}
\item When $0<k<t_1/s$, since $t+ks-1>10ks$, again we use Lemma \ref{lem:carind}
\begin{align*}
& R_P(\bm 1_{u+\log d(y,x)\in[-1,a+1],d(y,x')\leq e^{-s},d(y,x)\in[e^{-(k+1)s},e^{-ks}]})(x_o,t)\\
\leq & R_P(\bm 1_{d(y,x)\leq e^{-ks},u\in[-1+ks,a+1+(k+1)s]})(x_o,t)\lesssim (1+s+a)^2e^{-\epsilon'ks}.
\end{align*}
\item In the last case, $\log d(y,x)\leq -t_1$, we have
\begin{align*}
& R_P(\bm 1_{u+\log d(y,x)\in[-1,a+1],d(y,x')\leq e^{-s},d(y,x)\leq e^{-t_1}})(x_o,t)\\
\leq & R_P(\bm 1_{u+\log d(y,x)\in[-1,a+1],d(y,x)\leq e^{-t_1}})(x_o,t).
\end{align*}
This is similar to the original quantity $R_P(\bm 1_{u+\log d(y,x)\in[-1,a+1],d(y,x')\leq e^{-s}})(x_o,t)$. The difference is that here $t_1$ is comparable with $t$, which is crucial in the following argument. Return to the definition of $R_P$, and discuss on the length $n=\omega(g)$.
\begin{itemize}
\item 
When $n>(t+a+1)/(\sigma_{\mu}-2\epsilon)$, by inequality \eqref{ineq:lardev1} and \eqref{ineq:regcon}, we have 
$\check\mu^{\otimes n}\{g\in G^{\times n}|\kappa(g)-n\sigma_{\mu}\leq n\epsilon,\,d(gx_o,x)\geq e^{-\epsilon n} \}>1-Ce^{-\epsilon' n}$. By hypothesis $n>(t+a+1)/(\sigma_{\mu}-2\epsilon)$
, the element in this set satisfies \[\kappa(g)\geq(\sigma_{\mu}-\epsilon)n> t+a+1+n\epsilon\geq t+a+1-\log d(gx_o,x).\] 
Thus $\check\mu^{\otimes n}\{g\in G^{\times n}|\kappa(g)\in [t-1,t+a+1]-\log d(gx_o,x) \}=\oexp(n)$. Summing over $n$, we see that the measure of this part is $\oexp(t)$.
\item When $n\in[(t-1)/(\sigma_\mu+\epsilon),(t+a+1)/(\sigma_\mu-2\epsilon)]$, since $\epsilon n\geq \epsilon(t-1)/(\sigma_\mu+\epsilon)>(t-1)/9=t_1$, Corollary \ref{cor:regcon} implies that
\[\check{\mu}^{\otimes n}\{g\in G^{\times n}|d(gx_o,x)\leq e^{-t_1} \}=\oexp(t_1)=\oexp(t). \]
\item When $n\leq (t-1)/(\sigma_\mu+\epsilon)$, Corollary \ref{cor:lardev1} implies the measure of this part is $\oexp(t)$.
\end{itemize}
Therefore we have
\[R_P(\bm 1_{u+\log d(y,x)\in[-1,a+1],d(y,x)\leq e^{-t_1}})(x_o,t)=\oexp(t_1). \] 
\end{itemize}
Combining the three cases, we have finished the proof.
\end{proof}

\subsection{Residue process}
\label{subsecrespro}
We introduce the residue process, which not only deals with $\sigma(g_ng_{n-1}\cdots g_1,x)$ but also takes into account the next step $\sigma(g_{n+1},g_ng_{n-1}\cdots g_1x)$. 
Let $f$ be a positive bounded Borel function on $X\times\bb R^2$. For $(x,t)\in X\times\bb R$, we define the residue operator by
\begin{equation}
\begin{split}
	\Res f(x,t)&=\sum_{n\geq 0}\int f(hgx,\sigma(h,gx),\sigma(g,x)-t)\dd\mu^{*n}(g)\dd\mu(h).
\end{split}
\end{equation}
Let $\cal F_uf(x,v,\xi)=\int f(x,v,u)e^{iu\xi}\dd u$ be the Fourier transform on $\bb R_u$. Let $F$ be a function on $X\times \bb R_{v}\times\bb R_{\xi}$,. Define a partial Lipschitz norm by
\[|F|_\Lip=\sup_{\xi\in\bb R}|F(\xi)|_{Lip}=
\sup_{\xi\in\bb R}\left(|F(\xi)|_\infty+\sup_{(x,v),(x',v')\in X\times\bb R}\frac{|F(x,v,\xi)-F(x',v',\xi)|}{d(x,x')+|v-v'| }\right).  \]

\begin{prop}[Residue process]\label{prop:residue}
	If $f$ is a positive bounded continuous function on $X\times\bb R^2$. Assume that the projection of $\supp\cal F_u(f)$ onto $\bb R_\xi$ is contained in a compact set $K$, and $|\cal F_u(f)|_\Lip,|\partial_\xi\cal F_u(f)|_\Lip$ are finite. 
	Then for $t>0$ and $x\in X$, we have 
	\begin{equation}
	\begin{split}
		\Res f(x,t)=&\frac{1}{\sigma_\mu}\int_{-t}^\infty\int_G \int_X f(hy,\sigma(h,y),u)\dd\nu(y)\dd\mu(h)\dd u\\
		&+\frac{1}{t}O_K\left(|\cal F_u(f)|_\Lip+|\partial_\xi\cal F_u(f)|_\Lip \right).
		\end{split}
	\end{equation}
\end{prop}
\begin{proof}
	For a bounded continuous function $f$ on $X\times \bb R^2$ and $(x,u)\in X\times\bb R$, we define an operator $Q$ by
	\[Qf(x,u)=\int_G f(hx,\sigma(h,x),u)\dd\mu(h). \]
	Then
	\[\Res f(x,t)=\sum_{n\geq 0}\int Qf(gx,\sigma(g,x)-t)\dd\mu^{*n}(g)=R(Qf)(x,t). \]
	
	We want to use Proposition \ref{prop:renreg}, so we need to verify the hypotheses. The function $Qf$ is bounded and integrable by the hypotheses on $f$. Then
	 \begin{align*}
	 	\widehat{Qf}(x,\xi)&=\int Qf(x,u)e^{iu\xi}\dd u=\int f(hx,\sigma(h,x),u)e^{iu\xi}\dd u\dd\mu(h)\\
	 	&=\int_G\cal F_uf(hx,\sigma(h,x),\xi)\dd\mu(h).
	 \end{align*}
	 Thus $\widehat{Qf}$ is also compactly supported on $\xi$. It remains to estimate the H\"{o}lder norm of $\widehat{Qf}$.
	 Since $\cal F_uf(x,v,\xi)$ is Lipschitz on $(x,v)\in X\times \bb R$, this implies that
	 \begin{align*}
	 	|\widehat{Qf}(x,\xi)-&\widehat{Qf}(y,\xi)|\leq \int_G|\cal F_uf(hx,\sigma(h,x),\xi)-\cal 
	 	F_uf(hy,\sigma(h,y),\xi)|\dd\mu(h)\\
	 	&\leq \int_G |\cal F_u f|_\Lip(d(hx,hy) +|\sigma(h,x)-\sigma(h,y)| )\dd\mu(h).
	 \end{align*}
	 Using Lipschitz property of the distance and the cocycle, and finite exponential moment, we have
	 \begin{align*}
	 	|\widehat{Qf}(x,\xi)-&\widehat{Qf}(y,\xi)|\leq |\cal F_u f|_\Lip d(x,y)^\gamma\int_G (1+\kappa(h))\|h\|^{2\gamma}\dd\mu(h)\lesssim |\cal F_u f|_\Lip d(x,y)^\gamma,
	 \end{align*}
	  where we use the Remark \ref{rem:gameps} that $4\gamma\leq \epsilon_1$. Therefore
	 \begin{lem}[Change of norm]
	 Under the assumptions of Proposition \ref{prop:residue}, we have \[|\widehat{Qf}|_{\linf}\lesssim|\cal F_u(f)|_\Lip,\  |\partial_\xi\widehat{Qf}|_{\linf}\lesssim|\partial_\xi\cal F_uf|_\Lip.\]
	 \end{lem}
	 \begin{proof}
	 The second inequality follows by the same computation.
	 \end{proof}
	 By Proposition \ref{prop:renreg}, we have
	 	\begin{align*}
	 		R(Qf)(x,t&)=\frac{1}{\sigma_\mu}\int_X\int_{-t}^{\infty}Qf(y,u)\dd u\dd\nu(y)+\frac{1}{t}O_K\left(|\widehat{Qf}|_{\linf}+|\partial_\xi \widehat{Qf}|_{\linf}\right)\\
	 		&=\frac{1}{\sigma_\mu}\int_X\int_{-t}^{\infty}Qf(y,u)\dd u\dd\nu(y)+\frac{1}{t}O_K\left(|\cal F_u(f)|_\Lip+|\partial_\xi \cal F_u(f)|_\Lip\right).
	 	\end{align*}
	 The proof is complete.
\end{proof}

\subsection{Residue process with cutoff}
In this section, we restrict the residue process to the sequences $(g_{n+1},g_n,\dots,g_1)$ such that $\sigma(g_n\cdots g_1,x)<t\leq \sigma(g_{n+1}\cdots g_1,x)$. Let $f$ be a function on $X\times\bb R^2$. Define a Lipschitz norm by
\begin{align}
	| f|_{\lf}=| f|_\infty+\sup_{(x,v,u)\neq (x',v',u')}\frac{| f(x,v,u)-f(x',v',u')|}{d(x,x') +|v-v'| +|u-u'| }.
\end{align}
Define an operator from bounded Borel functions on $X\times\bb R^2$ to functions on $X\times\bb R$ by
\[\Cut f(x,t)=\sum_{n\geq 0}\int_{\sigma(g,x)<t\leq\sigma(hg,x)} f(hgx,\sigma(h,gx),\sigma(g,x)-t)\dd\mu(h)\dd\mu^{*n}(g). \]
By Lemma \ref{lem:resfin}, which will be proved later, this operator is well defined. Let $K$ be a compact set in $\bb R$. We denote $|K|$ by the supremum of the distance between a point in $K$ and $0$. 
\begin{prop}\label{prop:rescut}
Let $f$ be a continuous function on $X\times\bb R^2$ with $| f|_\lf$ finite. Assume that the projection of $\supp f$ on $\bb R_v$ is contained in a compact set $K$.
For all $\delta>0$, $t>|K|+\delta$ and $x\in X$, we have
	\begin{equation}
		\begin{split}
		\Cut f(x,t)=\int_X\int_G\int_{-\sigma(h,y)}^{0} f(hy,\sigma(h,y),u)\dd u\dd\mu(h)\dd\nu(y)+O_K(\delta +O_\delta/t)| f|_\lf,
		\end{split}
	\end{equation}
	where $O_K$ does not depend on $\delta, f,t,x$, and the integral $\int_{-\sigma(h,y)}^{0}=0$ if $\sigma(h,y)<0$.
\end{prop}
\begin{rem}
We decompose $f$ into real and imaginary parts, then decompose these two parts into positive and negative parts. Each part satisfies the hypotheses of Proposition \ref{prop:rescut}, with the support and the Lipschitz norm bounded by the original one. Thus, it is sufficient to prove this proposition for $f$ positive.
\end{rem}
The following lemma connects the operator $E_c$ with $E$.
\begin{lem}
Under the assumptions of Proposition \ref{prop:rescut}, let
$ f_o(x,v,u)=\bm 1_{-v\leq u<0} f(x,v,u)$. Then \[\Cut f(x,t)=\Res f_o(x,t).\]
\end{lem}
Before proving this proposition, we describe some regularity and independence properties. They are corollaries of analogous properties for the renewal process. The idea is to decompose the integral according to the last letter. The following lemma means that the residue process with cutoff has exponential decay with respect to the last jump.
\begin{lem}\label{lem:reslar}
	For $t,s$ in $\bb R$ and $x$ in $X$, we have
	\begin{align}
	\Cut(\bm 1_{v\geq s})(x,t)=\Res(\bm 1_{-v\leq u<0, v\geq s})(x,t)=\oexp(s).
	\end{align}
\end{lem}
\begin{proof}
	By Lemma \ref{lem:renintts} and finiteness of the exponential moment, we have
	\begin{align*}
	&\quad\sum_{n\geq 0}\mu\otimes\mu^{*n}\{(h,g)\in G^2|\sigma(g,x)-t\in[-\sigma(h,gx),0],\sigma(h,gx)\geq s\}\\
	&\leq\sum_{n\geq 0}\mu\otimes\mu^{*n}\{(h,g)\in G^2|\sigma(g,x)-t\in[-\kappa(h),0],\kappa(h)\geq s \}\\
	&=\int_{\kappa(h)>s}R(\bm 1_{[-\kappa(h),0]})(x,t)\dd\mu(h)\lesssim \int_{\kappa(h)>s}\max\{1,\kappa(h)\}\dd\mu(h)=\oexp(s).
	\end{align*}
	The proof is complete.
\end{proof}
\begin{lem}\label{lem:resfin}
There exists $C>0$ such that for all $t\in\bb R$ and $x\in X$, we have
	\begin{equation}\label{ineq:resfin}
		\Cut(\bm 1)(x,t)=\Res(\bm 1_{-v\leq u<0})(x,t)\leq C.
	\end{equation}
\end{lem}
This is a special case of Lemma \ref{lem:reslar}.
The following lemma quantifies the independence of the scalar part and the angle part. Abbreviate $1_{d(y,x')\leq e^{-s},-v\leq u<0}(y,v,u)$ to $1_{d(y,x')\leq e^{-s},-v\leq u<0}$, and others are similar.
\begin{lem}\label{lem:resind}
For $t>5s>0$ and $x,x'$ in $X$, we have
	\begin{equation}
		\Cut(\bm 1_{d(y,x')\leq e^{-s}})(x,t)=\Res (\bm 1_{d(y,x')\leq e^{-s},-v\leq u<0})(x,t)=\oexp(s).
	\end{equation}
\end{lem}
\begin{proof}
	Since 
	\[\bm 1_{-v\leq u<0}\leq \bm 1_{d(y,x')\leq e^{-s},-v\leq u<0, v\geq s}+\bm 1_{d(y,x')\leq e^{-s},0\leq u+v<s}, \]
	we have
	\begin{align*}
		\Res(\bm 1_{d(y,x')\leq e^{-s},-v\leq u<0})(x,t)&\leq \Res(\bm 1_{-v\leq u<0, v\geq s})(x,t)+\Res(\bm 1_{d(y,x')\leq e^{-s},0\leq u+v<s})(x,t).
	\end{align*}
	By definition, we have 
	\begin{align*}
	\Res(\bm 1_{d(y,x')\leq e^{-s},0\leq u+v<s})(x,t)&=\sum_{n\geq 0}\int\bm 1_{d(hgx,x')\leq e^{-s},\sigma(h,gx)+\sigma(g,x)-t\in[0,s]}\dd\mu^{*n}(g)\dd\mu(h)\\
	&=\sum_{n\geq 1}\int\bm 1_{d(gx,x')\leq e^{-s},\sigma(g,x)-t\in[0,s]}\dd\mu^{*n}(g)= R(\bm 1_{B(x',e^{-s}),[0,s]})(x,t).
	\end{align*}
	By Lemma \ref{lem:reslar} and Proposition \ref{prop:cocind}, the result follows.
\end{proof}
\begin{lem}\label{lem:resindinv}
	For $s>0$, $t>\max\{10s,10\} $ and $x,x_o,x'\in X$, we have
	\begin{align*}
	\sum_{n\geq 0}\mu\otimes\mu^{*n}\{(h,g)\in G\times G|\sigma(hg,x)\geq t,\sigma(g,x)< t, d((hg)^{-1}x_o,x')\leq e^{-s} \}=\oexp(s).
	\end{align*}
\end{lem}
By the same argument as in the proof of Lemma \ref{lem:resind}, we only need to replace Proposition \ref{prop:cocind} by Proposition \ref{prop:cocindinv}. The difference between this lemma and Lemma \ref{lem:resind} is the angle part $(hg)^{-1}x$.

Using $\psi_\delta$ to regularize these functions, we write $ f_\delta(x,v,u)=\int  f_o(x,v,u-u_1)\psi_\delta(u_1)\dd u_1=\psi_\delta*f_o(x,v,u)$.
\begin{lem}\label{lem:rescut}
Under the same hypotheses as in Proposition \ref{prop:rescut}, we have
	\begin{align*}
		\Res( f_\delta)(x,t)=\int_X\int_G\int_{-\sigma(h,y)}^{0} f(hy,\sigma(h,y),u)\dd u\dd\mu(g)\dd\nu(y)+O_K(\delta+\frac{O_\delta}{t})| f|_\lf.
	\end{align*}
\end{lem}
\begin{proof}
We want to verify the conditions in Proposition \ref{prop:residue} and then use this proposition. The integrable condition is valid because $|\int_{\bb R_u} f_\delta|=|\int_{\bb R_u} f_o(x,v,u)\dd u|=|\int_{-v}^{0}f(x,v,u)\dd u|\leq |K|| f|_\infty$. For the Fourier transform, we have
\begin{align*}
	\cal F_u f_\delta=\cal F_u(\psi_\delta*f_o)=\hat{\psi}_\delta\cal F_u f_o.
\end{align*} 
We need to estimate the Lipschitz norm of $\cal F_u f_o$. This function equals
\[\int  f_o(x,v,u)e^{i\xi u}\dd u=\int_{-v}^{0} f(x,v,u)e^{i\xi u}\dd u. \] 
Taking $(x,v)\neq(x',v')$, we have
\begin{align*}
	&|\int_{-v}^{0} f(x,v,u)e^{i\xi u}\dd u-\int_{-v'}^{0} f(x',v',u)e^{i\xi u}\dd u|\\
	\leq &|\int_{-v}^{0}( f(x,v,u)- f(x',v',u))e^{i\xi u}\dd u|+|v'-v||f|_\infty\lesssim |K||f|_\ck(d(x,x') +|v-v'| ).
\end{align*}
Then we have 
\begin{lem}[Change of norm] Under the same hypotheses as in Proposition \ref{prop:rescut}, we have
\[|\cal F_u f_\delta|_\Lip\leq |K||f|_\ck,\ |\partial_\xi\cal F f_{\delta}|_\Lip\leq |K|^2|f|_\ck.\]
\end{lem}
\begin{proof}
Noting that in the integration $|u|\leq |v|$, we get the second inequality by the same computation.
\end{proof}

Therefore by Proposition \ref{prop:residue}, we have
\begin{align*}
\Res( f_\delta)(x,t)=\frac{1}{\sigma_\mu}\int_{-t}^\infty\int_G \int_X  f_\delta(hy,\sigma(h,y),u)\dd\nu(y)\dd\mu(h)\dd u+\frac{O_\delta}{t}\left(|f|_\ck(|K|+|K|^2)\right).
\end{align*}
Then
\begin{align*}
	\int_{-t}^{\infty} f_{\delta}(x,v,u)\dd u&=\int_{-t}^{\infty}\int_{-v}^{0} f(x,v,u_1)\psi_\delta(u-u_1)\dd u_1\dd u=\int_{-v}^{0} f(x,v,u_1)\int_{-t}^{\infty}\psi_{\delta}(u-u_1)\dd u\dd u_1\\
	&=\int_{-v}^{0} f(x,v,u_1)\dd u_1-\int_{-v}^{0} f(x,v,u_1)\int_{-\infty}^{-t-u_1}\psi_{\delta}(u)\dd u\dd u_1.
\end{align*}
Since $t-\delta\geq |K|$, we have $-t-u_1\leq -t+v\leq -\delta$. By $\int_{-\infty}^{-\delta}\psi_\delta\leq C_\psi\delta$, this implies that $\int_{-t}^{\infty} f_{\delta}(x,v,u)\dd u=\int_{-v}^{0} f_{\delta}(x,v,u)\dd u(1+O(\delta))$.
Using Lemma \ref{lem:resfin}, we have
\[|\int_X\int_G\int_{-\sigma(h,y)}^{0} f(hy,\sigma(h,y),u)\dd u\dd\mu(g)\dd\nu(y)|\leq | f|_\infty \Cut(\bm 1)=O(| f|_\infty). \]
Therefore
\begin{align*}
&\int_{-t}^\infty\int_G \int_X  f_\delta(hy,\sigma(h,y),u)\dd\nu(y)\dd\mu(h)\dd u\\
&\qquad\qquad=\int_X\int_G\int_{-\sigma(h,y)}^{0} f(hy,\sigma(h,y),u)\dd u\dd\mu(g)\dd\nu(y)+O(\delta|f|_\infty).
\end{align*}
The proof is complete.
\end{proof}
Next lemma gives the difference between a function and its regularization.
\begin{lem}\label{lem:chadif} Let $\varphi_0(u)=\bm 1_{[b_1,b_2]}(u)\varphi(u)$, where $b_2>b_1$ and $|\varphi'|_{L^{\infty}}< \infty$, $|\varphi|_{L^\infty}\leq 1$. Then we have
	\begin{equation}\label{ineq:chadif}
	|\psi_{\delta}*\varphi_0(u)-\varphi_0(u)|\leq
	\begin{cases}
	(|\varphi'|_\infty+2)\delta & u\in[b_1+\delta,b_2-\delta],\\
	2 & u\in[b_1-\delta,b_1+\delta]\cup[b_2-\delta,b_2+\delta],\\
	\psi_{\delta}*\bm 1_{[b_1,b_2]}(u) & u\in [b_1-\delta,b_2+\delta]^c.
	\end{cases}
	\end{equation}
\end{lem}
\begin{proof} We will prove this inequality in each interval.
	\begin{itemize}
		\item When $u$ is in $[b_1+\delta,b_2-\delta]$, we have
		\begin{align*}
		|(\psi_{\delta}*\varphi_0-\varphi_0)(u)|=\left|\int\psi_{\delta}(t)(\varphi_0(u-t)-\varphi_0(u))\dd t\right|\leq\int_{-\delta}^{\delta}\psi_{\delta}(t)|\varphi_0(u-t)-\varphi_0(u)|\dd t+2\delta.
		\end{align*}
		When $|t|\leq \delta$, we have $u-t\in[b_1,b_2]$. Since $|\varphi_0'(u)|\leq|\varphi'|_\infty$ for $u\in[b_1,b_2]$, this implies that
		\[\int_{-\delta}^{\delta}\psi_{\delta}(t)|\varphi_0(u-t)-\varphi_0(u)|\dd t
		\leq\int_{-\delta}^{\delta}\psi_{\delta}(t)|t||\varphi'|_\infty\dd t\leq \delta|\varphi|_\infty. \]
		\item When $u\in[b_1-\delta,b_1+\delta]\cup[b_2-\delta,b_2+\delta]$, we use the trivial bound $|\psi_{\delta}*\varphi_0(u)-\varphi_0(u)|\leq 2$.
		\item When $u\in(-\infty,b_1-\delta]\cup[b_2+\delta,\infty]$, we have $\varphi_0(u)=0$, then $|\psi_{\delta}*\varphi_0|\leq|\psi_{\delta}*\bm 1_{[b_1,b_2]}|$.
	\end{itemize}
	Thus collecting all together, we get the inequality.
\end{proof}
\begin{proof}[Proof of Proposition \ref{prop:rescut}]
	To simplifier the notation, we normalize $ f$ in such a way that $| f|_\infty=1$. By Lemma \ref{lem:rescut}, we only need to give an estimate of $\Res(| f_\delta- f_o|)(x,t)$. 
	
	Since $ f_o(x,v,u)=\bm 1_{-v\leq u<0}(u) f(x,v,u)$ with $(x,v)$ fixed, Lemma \ref{lem:chadif} implies that
	\begin{equation*}
		| f_\delta- f_o|(u)\leq
		\begin{cases}
		(|\partial_u f|_\infty+2)\delta & u\in[-v+\delta,-\delta],\\
		2 & u\in[-v-\delta,-v+\delta]\cup[-\delta,\delta],\\
		\psi_\delta*\bm 1_{[-v,0]}(u) & u\in[-v-\delta,\delta]^c.
		\end{cases}	
	\end{equation*}
	By definition of $|K|$, the first term is less than $(|\partial_u f|_\infty+2)\delta\bm 1_{[-|K|+\delta,-\delta]}$. The third term equals
	\begin{align*}
		\bm 1_{[-\infty,-v-\delta]\cup[\delta,\infty]}&\psi_\delta*\bm 1_{[-v,0]}(u)=\bm 1_{[-\infty,-v-\delta]\cup[\delta,\infty]}(u)\int_{-v}^0\psi_\delta(u-u_1)\dd u_1\\
		&=\bm 1_{[-\infty,-v-\delta]\cup[\delta,\infty]}(u)\int_{u}^{u+v} \psi_\delta(u_1)\dd u_1.
	\end{align*}
	
	By definition  and the above arguments, we have
	\begin{align*}
		\Res(| f_\delta- f_o|)(x,t)&=\sum_{n\geq 0}\int | f_\delta- f_o|(hgx,\sigma(h,gx),\sigma(g,x)-t)\dd\mu^{*n}(g)\dd\mu(h)\\
		& \leq\sum_{n\geq 0}\int\Big((|\partial_u f|_\infty+2)\delta\bm 1_{[-|K|,-\delta]}(\sigma(g,x)-t) +\\
		&+2\bm 1_{[-\sigma(h,gx)-\delta,-\sigma(h,gx)+\delta]\cup[-\delta,\delta]}(\sigma(g,x)-t)\\
		&+\bm 1_{[-\infty,-\sigma(h,gx)-\delta]\cup[\delta,\infty]}(\sigma(g,x)-t)\int_{\sigma(g,x)-t}^{\sigma(hg,x)-t}\psi_\delta(u_1)\dd u_1\Big) \dd\mu^{*n}(g)\dd\mu(h).
	\end{align*}
	By Lemma \ref{lem:renintts}, the first term is controlled by $(|\partial_u f|_\infty+2)\delta |K|$. The second term is less than $R(\bm 1_{[-\delta,\delta]})(x,t)$. Due to Proposition \ref{prop:renint}, it is controlled by $6\delta(1/\sigma_\mu+C_\delta(1+2\delta)/t)$. 
	
	For the third term, we need to change the order of integration. Since $\sigma(g,x)-t>\delta$ or $\sigma(g,x)-t<-\sigma(h,gx)-\delta$, we have $u_1\geq\sigma(g,x)-t>\delta$ or $u_1\leq \sigma(hg,x)-t=\sigma(h,gx)+\sigma(g,x)-t\leq -\delta$. We integrate first with respect to $u_1$, then the third term is less than
	\[\int_{[-\infty,-\delta]\cup[\delta,\infty]}\psi_\delta(u_1)\sum_{n\geq 0}\mu\otimes\mu^{*n}\{(h,g)|\sigma(hg,x)\geq u_1+t,\sigma(g,x)\leq u_1+t \}\dd u_1. \]
	By Lemma \ref{lem:resfin}, the above quantity is less than $C\int_{[-\infty,-\delta]\cup[\delta,\infty]}\psi_\delta(u_1)\dd u_1\lesssim \delta$.
	
	Therefore, we have
	\begin{align*}
	\Res(| f_\delta- f|)(x,t)=O_K(\delta+C_\delta/t)| f|_\lf.
	\end{align*}
	The proof is complete.
\end{proof}
\begin{rem}[Minus case]
The lemmas in this part concern plus and minus. The another version we need is for $\Cut^-( f)(x,t)=\Res(\bm 1_{0< u\leq -v} f)(x,t)$, the proofs are exactly the same. 
\begin{repprop}{prop:rescut}
	Under the assumptions of Proposition \ref{prop:rescut}, we have
\[\Cut^-( f)(x,t)=\int_X\int_G\int^{-\sigma(h,y)}_{0} f(hy,\sigma(h,y),u)\dd u\dd\mu(h)\dd\nu(y)+O_K(\delta +O_\delta/t)| f|_\lf .\]
\end{repprop}
\end{rem}

\subsection{Residue process for the Cartan Projection}
We consider the residue process for the cutoff of a function $f$ on $X^2\times\bb R^2$, where the cocycle is replaced by the Cartan projection. We will give a limit not only with $gx$, but also with $g^{-1}x'$.

As in the previous subsection, we can define a similar Lipschitz norm on the space of Lipschitz functions on $X^2\times\bb R^2$, using the same name $|f|_{\ck}$.
Define the operator from bounded Borel functions on $X^2\times\bb R^2$ to functions on $X^2\times\bb R$ by
\[\Car f(x',x,t)=\sum_{n\geq 0}\int_{\kappa(g)< t\leq \kappa(hg) } f((hg)^{-1}x',hgx,\kappa(hg)-\kappa(g),\kappa(g)-t)\dd\mu(h)\dd\mu^{*n}(g). \]

\begin{prop}\label{prop:rescar} Let $ f$ be a continuous function on $X^2\times\bb R^2$ with $| f|_\lf$ finite. Assume that the projection of $\supp f$ on $\bb R_v$ is contained in a compact set $K$. For all $\delta>0$, $t>\max\{2(|K|+\delta),20\}$ and $x',x$ in $X$, we have
	\begin{equation}
	\begin{split}
	\Car f(x',x,t)=\int_{X^2}\int_G\int_{-\sigma(h,y)}^{0} f(y',hy,\sigma(h,y),u)\dd u\dd\mu(h)\dd\nu(y)\dd\check\nu(y')+O_K(\delta+O_\delta/t)| f|_\lf,
	\end{split}
	\end{equation}
	where $O_K$ does not depend on $\delta, f,t,x,x'$, the integral $\int_{-\sigma(h,y)}^{0}=0$ if $\sigma(h,y)<0$.
\end{prop}
\begin{proof}
We introduce local notations here: for an element $g$ in $G$ and a continuous function $f$ on $X^2\times \bb R^2$, define $gf(x',x,v,u)=f(g^{-1}x',x,v,u)$. Let $f_{x'}(x,v,u)=f(x',x,v,u)$, which emphasizes that the first coordinate is fixed. Let $l=[\epsilon_5t/\sigma_\mu], \text{ where } \epsilon_5<1/10$.
We use the decomposition 
$$h=g_{n+1}, g=(g_n,\dots ,g_{l+1}), \glg=(g_l,\dots, g_1).$$ Recall that $N_t^+=\bigcup_{n\geq 0}\{(g_{n+1},g_n,\dots ,g_1)|\kappa(g_{n+1}\cdots g_1)\geq t>\kappa(g_n\cdots g_1) \}$. Let $N_t^+(n)=N_t^+\cap G^{\times(n+1)}= \{(g_{n+1},\dots,g_1)|\kappa(g\glg)\leq t,\kappa(hg\glg)>t \}$. Let
\begin{align*}
T_{n}(x,t)&=\{(g_{n+1},\dots,g_1)\in G^{\times(n+1)}|\sigma(hg,\glg x)>t-\kappa(\glg),\sigma(g,\glg x)\leq t-\kappa(\glg) \},
\end{align*}
and let $G_{\epsilon,l}=\{(g_l,\dots,g_1)||\kappa(\glg)-l\sigma_\mu|\leq l\epsilon, d(x_{g_l\cdots g_1}^M,x')\geq e^{-\epsilon l} \}$, as well as 
\[T_{n,\epsilon}=\{(g_{n+1},\dots,g_1)\in T_n|(g_l,\dots,g_1)\in G_{\epsilon,l} \}.\]
\textbf{Step 1:} Due to Corollary \ref{cor:lardev} and Corollary \ref{cor:lardev1}, the sum of the integrals $\int_{N_t^+(n)}$ for $n$ ranging from  $t/(\sigma_\mu+\epsilon)-1$ to $ t/(\sigma_\mu-\epsilon)$ is exponentially small in $t$. In other words, we have
\begin{equation}\label{ineq:fcar}
\begin{split}
	|\sum_{n=[t/(\sigma_{\mu}+\epsilon)]}^{[t/(\sigma_\mu-\epsilon)]}\int_{N_t^+(n)} f((hg\glg)^{-1}x',hg\glg x,\kappa(hg\glg)-\kappa(g\glg),\kappa(g\glg)-t)\dd\mu^{\otimes(n+1)}\\
	-\Car f(x',x,t) |=\oexp(t)|f|_\infty
\end{split}
\end{equation}
The following lemma replaces the Cartan projection with the cocycle.
\begin{lem}\label{lem:jff}
Under the same assumption as in Proposition \ref{prop:rescar}, we have
\begin{equation}\label{ineq:jff}
\begin{split}
&|\sum_{n=[t/(\sigma_{\mu}+\epsilon)]}^{[t/(\sigma_\mu-\epsilon)]}\int_{N_t^+(n)} f((hg\glg)^{-1}x',hg\glg x,\kappa(hg\glg)-\kappa(g\glg),\kappa(g\glg)-t)\dd\mu^{\otimes(n+1)}(hg\glg)\\
&-\sum_{n=[t/(\sigma_{\mu}+\epsilon)]}^{[t/(\sigma_\mu-\epsilon)]}\int_{T_{n,\epsilon}}\glg f((hg)^{-1}x',hg\glg x,\sigma(h,g\glg x),\sigma(g,\glg x)-(t-\kappa(\glg)))\dd\mu^{\otimes(n+1)}(hg\glg)|\\
&=O(\delta+O_\delta/t)|f|_\ck.
\end{split}
\end{equation}
\end{lem}
This lemma will be proved later. We will decompose $T_{n,\epsilon}(x,t)$ to apply the residue process for the cocycle. The space $T_{n,\epsilon}(x,t)$ can be seen as a fibered space over $G_{\epsilon,l}$. When the first $l$ elements are fixed, the elements $(h,g)$ such that $hg\glg=(g_{n+1},\dots, g_1)\in T_{n,\epsilon}(x,t)$, are the admitted elements in the residue process with cutoff, whose start point is $\glg x$ and time is $t-\kappa(\glg)$. 
Since $(n-l)(\sigma_\mu+\epsilon)\leq t-\kappa(\glg)$ and $(n-l)(\sigma_\mu-\epsilon)\geq t-\kappa(\glg)$, we can apply Principle 1 to this residue process. Integrating over $G_{\epsilon,l}$ implies that
\begin{equation}\label{ineq:finv}
\begin{split}
|\sum_{n=[t/(\sigma_{\mu}+\epsilon)]}^{[t/(\sigma_\mu-\epsilon)]}\int_{T_{n,\epsilon}}\glg f((hg)^{-1}x',hg\glg x,\sigma(h,g\glg x),\sigma(g,\glg x)-(t-\kappa(\glg)))\dd\mu^{\otimes(n+1)}(hg\glg)\\
-\int_{G_{\epsilon,l}}\Inv \glg f(x',jx,t-\kappa(j))\dd\mu^{\otimes l}(\glg)|=\oexp(t)|f|_\infty.
\end{split}
\end{equation}
where
\[\Inv f(x',x,t)=\sum_{n\geq 0}\int_{\sigma(g,x)<t\leq\sigma(hg,x)} f((hg)^{-1}x',hgx,\sigma(h,gx),\sigma(g,x)-t)\dd\mu(h)\dd\mu^{*n}(g). \]
The following inequality, whose proof relies on Lemma \ref{lem:resindinv}, will give a major term.
\begin{lem}\label{lem:ecf}
Under the same assumption as in Proposition \ref{prop:rescar}, for all $\glg\in G_{\epsilon,l}$, we have
\begin{equation}\label{ineq:ecf}
\begin{split}
|&\Cut f_{\glg^{-1}x'}(\glg x,t-\kappa(\glg))-\Inv \glg f(x',\glg x,t-\kappa(\glg))|\leq |f|_\ck \oexp(l),
\end{split}
\end{equation}
\end{lem}
This lemma will be proved later. Integrating \eqref{ineq:ecf} over $G_{\epsilon,l}$, we obtain 
\begin{equation}\label{ineq:cutinv}
\begin{split}
|\int_{G_{\epsilon,l}}\Cut f_{\glg^{-1}x'}(\glg x,t-\kappa(\glg))-\Inv \glg f(x',\glg x,t-\kappa(\glg))\dd\mu^{\otimes l}(\glg)|\leq |f|_\ck\oexp(t).
\end{split}
\end{equation}
By \eqref{ineq:fcar}\eqref{ineq:jff}\eqref{ineq:finv},  it suffices to compute the major term 
\[\int_{G_{\epsilon,l}}\Cut f_{\glg^{-1}x'}(\glg x,t-\kappa(\glg))\dd\mu^{\otimes}(\glg).\] 

\textbf{Step 2:} Recall that $N_0,P(0)$ are the two operators defined by $N_0\varphi=\int \varphi\dd\nu,\ P(0)\varphi(x)=\int \varphi(gx)\dd\nu(g)$, where $\varphi$ is a function in $\cal H^\gamma(X)$. We have another property of transfer operators \cite[Lemma 11.18]{benoistquint}: The spectral radius of $P=P(0)$ restricted to $\ker N_0$ is less than 1, which means that there exist $\rho<1,C>0$ such that for every function $\varphi$ in $\cal H^{\gamma}(X)$, we have
\begin{equation*}
|P^n\varphi-\int \varphi \dd\nu|_\infty\leq C\rho^n|\varphi |_{\gamma}.
\end{equation*}
Thus by $\mu^{\otimes l}G_{\epsilon,l}=\oexp(l)$, we have
\begin{equation}\label{ineq:gelgnu}
|\int_{G_{\epsilon,l}}\varphi (\glg^{-1}x)\dd{\mu}^{\otimes l}-\int \varphi \dd\check{\nu}|=|\int_{G^{\times l}}\varphi (\glg^{-1}x)\dd\mu^{\otimes l}(\glg)-\int \varphi \dd\check{\nu}|+\oexp(l)|\varphi |_\infty=\oexp(l)|\varphi |_{\ck}.
\end{equation}
By the definition of $|\cdot|_\ck$ on $X\times\bb R^2$, the function $f_{\glg^{-1}x'}(x,v,u)$ has a finite $|\cdot|_\lf$ value. Together with $t-\kappa(\glg)\geq t/2\geq |K|+\delta$, Proposition \ref{prop:rescut} implies that
\begin{equation}
\begin{split}
&\int_{G_{\epsilon,l}}\Cut f_{\glg^{-1}x'}(\glg x,t-\kappa(\glg))\dd\mu^{\otimes}(\glg)\\
&=\int_{G_{\epsilon,l}}\left(\int_{X}\int_G\int_{-\sigma(h,y)}^{0} f_{\glg^{-1}x'}(y,\sigma(h,y),u)\dd u\dd\mu(h)\dd\nu(y)\dd\mu^{\otimes l}(\glg)+O_K(\delta+O_\delta/t)|f_{\glg^{-1}x'}|_\lf\right)\\
&=\int_{X}\int_G\int_{-\sigma(h,y)}^{0} \int_{G_{\epsilon,l}}f(\glg^{-1}x',y,\sigma(h,y),u)\dd\mu^{\otimes l}(\glg)\dd u\dd\mu(h)\dd\nu(y)+O_K(\delta+O_\delta/t)|f|_\ck).
\end{split}
\end{equation}
With $(x,v,u)$ fixed, $f(x',x,v,u)$ is a Lipschitz function on $x'$, so it is a H\"older function. Together with Lemma \ref{lem:resfin} and inequality \eqref{ineq:gelgnu}, we have
\begin{equation}
\begin{split}
&\int_{G_{\epsilon,l}}\Cut f_{\glg^{-1}x'}(\glg x,t-\kappa(\glg))\dd\mu^{\otimes}(\glg)\\
&=\int_{X}\int_G\int_{-\sigma(h,y)}^{0} \int_{X}f(u,\sigma(h,y),y,y')\dd\check\nu(y')\dd u\dd\mu(h)\dd\nu(y)+(\oexp(l)+O_K(\delta+O_\delta/t))|f|_\ck.
\end{split}
\end{equation}
The result follows. 
\end{proof}
It remains to prove Lemma \ref{lem:jff} and Lemma \ref{lem:ecf}.
\begin{proof}[Proof of Lemma \ref{lem:jff}]
There exist $S_{n+1,l,x}\subset G^{\times (n+1)}$ and $S_{n,l,x}\subset G^{\times n}$ which satisfy the conditions in Lemma \ref{lem:coccar}. Let $S_n(x)=S_{n+1,l,x}\cap (G\times S_{n,l,x})$. Then
\begin{equation}\label{ineq:clocar}
\mu^{\otimes(n+1)}S_n(x)^c=\oexp(l),
\end{equation}
and for $(g_{n+1},\dots,g_1)$ in $S_n(x)$,
we have
\begin{align*}
	&|\kappa(hg\glg)-\sigma(hg,\glg x)-\kappa(\glg)|\leq e^{-\epsilon l}\\
	&|\kappa(g\glg)-\sigma(g,\glg x)-\kappa(\glg)|\leq e^{-\epsilon l}.
\end{align*}

In $N_t^+(n)\cap S_n(x)\cap T_n(x,t)$, we can replace the Cartan projection by the cocycle with exponentially small error. Fortunately, the difference of this set with $N_t^+(n)$ and $T_n(x,t)$ has exponentially small measure. By definition, we have
\[N_t^+(n)\cap S_n(x)\subset \{\sigma(hg,\glg x)>t-e^{-l\epsilon}-\kappa(\glg),\sigma(g,\glg x)\leq t+e^{-l\epsilon}-\kappa(\glg) \},\] 
and 
\[ N_t^+(n)\supset\{\sigma(hg,\glg x)>t+e^{-l\epsilon}-\kappa(\glg),\sigma(g,\glg x)\leq t-e^{-l\epsilon}-\kappa(\glg) \}\cap S_n(x).\]
Therefore
\begin{align*}
	(N_t^+(n)\cap S_n(x)-T_n(x,t))\subset&\{\sigma(hg,\glg x)\in[-e^{-\epsilon l},0]+t-\kappa(\glg) \}\\
	&\cup\{\sigma(g,\glg x)\in [0,e^{-\epsilon l}]+t-\kappa(\glg) \},
\end{align*}
and
\begin{align*}
	(T_n(x,t)\cap S_n(x)-N_t^+(n))\subset&\{\sigma(hg,\glg x)\in[0,e^{-\epsilon l}]+t-\kappa(\glg) \}\\
	&\cup\{\sigma(g,\glg x)\in [-e^{-\epsilon l},0]+t-\kappa(\glg) \} .
\end{align*}
Hence, these imply that
\begin{equation*}
\begin{split}
&\mu^{\otimes(n+1)}(N_t^+(n)-N_t^+(n)\cap S_n(x)\cap T_n(x,t))\leq\mu^{\otimes(n+1)}S_n(x)^c\\
&\qquad\qquad+ \mu^{\otimes(n+1)}(N_t^+(n)\cap S_n(x)-T_n(x,t))\\
&\leq \oexp(l)+\mu^{\otimes(n+1)}\{\sigma(hg,\glg x)\in[-e^{-\epsilon l},0]+t-\kappa(\glg) \}\cup\{\sigma(g,\glg x)\in [0,e^{-\epsilon l}]+t-\kappa(\glg) \}.
\end{split}
\end{equation*}
and
\begin{equation*}
\begin{split}
&\mu^{\otimes(n+1)}(T_n(x,t)-N_t^+(n)\cap S_n(x)\cap T_n(x,t))\leq\mu^{\otimes(n+1)}S_n(x)^c\\
&\qquad\qquad+ \mu^{\otimes(n+1)}(T_n(x,t)\cap S_n(x)-N_t^+(n))\\
&\leq \oexp(l)+\mu^{\otimes(n+1)}\{\sigma(hg,\glg x)\in[0,e^{-\epsilon l}]+t-\kappa(\glg) \}\cup\{\sigma(g,\glg x)\in [-e^{-\epsilon l},0]+t-\kappa(\glg) \}. 
\end{split}
\end{equation*}
Moreover, for $(g_{n+1},\dots,g_1)$ in the set $N_t^+(n)\cap S_n(x)\cap T_n(x,t)$, the definition of $S_n(x)$ implies that
\begin{equation}\label{ineq:clofkgm}
\begin{split}
	&|f((hg\glg)^{-1}x',hg\glg x,\kappa(hg\glg)-\kappa(g\glg),\kappa(g\glg)-t)\\
	&\qquad-\glg f((hg)^{-1}x',hg\glg x,\sigma(h,g\glg x),\sigma(g,\glg x)-(t-\kappa(\glg)))|
	\leq e^{-\gamma l\epsilon}|f|_\ck.
\end{split}
\end{equation}
Thus, for $n\in[t/(\sigma_\mu+\epsilon)-1,t/(\sigma_\mu-\epsilon)]$, we have
\begin{equation*}
	\begin{split}
	&|\int_{N_t^+(n)} f((hg\glg)^{-1}x',hg\glg x,\kappa(hg\glg)-\kappa(g\glg),\kappa(g\glg)-t)\dd\mu^{\otimes(n+1)}\\
	&-\int_{T_n(x,t)}\glg f((hg)^{-1}x',hg\glg x,\sigma(h,g\glg x),\sigma(g,\glg x)-(t-\kappa(\glg)))\dd\mu^{\otimes(n+1)}|\\
	&\leq \mu^{\otimes}(N_t^+(n)-N_t^+(n)\cap S_n(x)\cap T_n(x,t))\cup(T_n(x,t)-N_t^+(n)\cap S_n(x)\cap T_n(x,t))\\
	&\qquad +\mu^{\times(n+1)}{N_t^+(n)\cap S_n(x)\cap T_n(x,t)}\oexp(l)|f|_\ck\\
	&\leq (\oexp(l)+\mu^{\otimes(n+1)}\{|\sigma(hg,\glg x)-t+\kappa(\glg)|,
	|\sigma(g,\glg x)- t+\kappa(\glg)|\leq e^{-l\epsilon} \})|f|_\infty+\oexp(l)|f|_\ck.
	\end{split}
\end{equation*}
Sum up over all $n\in[t/(\sigma_\mu+\epsilon)-1,t/(\sigma_\mu-\epsilon)]$. Then the above inequality becomes
\begin{equation}\label{ineq:ntntn}
\begin{split} 
|\sum_{n=[t/(\sigma_{\mu}+\epsilon)]}^{[t/(\sigma_\mu-\epsilon)]}\int_{N_t^+(n)} f((hg\glg)^{-1}x',hg\glg x,\kappa(hg\glg)-\kappa(g\glg),\kappa(g\glg)-t)\dd\mu^{\otimes(n+1)}(hg\glg)|\\
-\sum_{n=[t/(\sigma_{\mu}+\epsilon)]}^{[t/(\sigma_\mu-\epsilon)]}\int_{T_n}\glg f((hg)^{-1}x',hg\glg x,\sigma(h,g\glg x),\sigma(g,\glg x)-(t-\kappa(\glg)))\dd\mu^{\otimes(n+1)}(hg\glg)\\
\leq t\oexp(l)|f|_\ck+|f|_{\infty}\int_{G^{\times l}}2 R(\bm 1_{[-e^{-\epsilon l},e^{-\epsilon l}]})(\glg x,t-\kappa(\glg))\dd\mu^{\otimes l}(\glg).
\end{split}
\end{equation}

By \eqref{ineq:lardev1}, \eqref{ineq:xgmx}, we have $\mu^{\otimes l}G_{\epsilon,l}\geq 1- \oexp(l)$. Thus combined with Lemma \ref{lem:resfin}, we get
\begin{align*}
\sum_{n\geq l}\mu^{\otimes(n+1)}(T_n(x,t)-T_{n,\epsilon}(x,t))=\int_{G_{\epsilon,l}^c}E\bm 1(\glg x,t-\kappa(\glg))\dd\mu^{\otimes l}(\glg)=\oexp(l).
\end{align*}
This enables us to replace the integration domain $T_n$ by $T_{n,\epsilon}$ with exponentially small error. It is sufficient to control the right hand side of \eqref{ineq:ntntn}.

The last term can be bounded by the similar argument as in \eqref{ineq:intren}, with Proposition \ref{prop:cocind} replaced by inequality \eqref{ineq:renintsma}. It follows that
\begin{equation}
\int_{G^{\times l}} 2R(\bm 1_{[-e^{-\epsilon l},e^{-\epsilon l}]})(\glg x,t-\kappa(\glg))\dd\mu^{\otimes l}(\glg)=\oexp(l)+\delta O(1+O_\delta/t).
\end{equation}
The proof is complete.
\end{proof}
\begin{proof}[Proof of Lemma \ref{lem:ecf}]
We want to replace $(hg\glg)^{-1}x'$ with $(\glg)^{-1}x'$ in the first coordinate in order to find the residue process with cutoff. The idea is always similar. We have a good approximation in a large set, whose complement has exponentially small measure. Let 
\[\Sigma_l=\bigcup_{n\geq 0}\{(h,g)\in G\times G^{\times n}|\sigma(g,\glg x)<t-\kappa(\glg)\leq \sigma(hg,jx), d((hg)^{-1}x',x^M_\glg)\leq e^{-\epsilon l } \}.\]
Since $t-\kappa(\glg)\geq t-(\sigma_\mu+\epsilon)l\geq 10\epsilon l$ and $t-\kappa(\glg)\geq t/2>10$, we can use Lemma \ref{lem:resindinv} with $s=\epsilon l$ and $jx,x',x_\glg^M$ to obtain
\begin{equation}\label{ineq:resindinvj}
\begin{split}
\mu\otimes\bar{\mu}\Sigma_l=\oexp(l).
\end{split}
\end{equation}
The definition of $G_{\epsilon,l}$ implies that $d(x_\glg^M,x')\geq e^{-\epsilon l}$ and $\kappa(\glg)\geq (\sigma_{\mu}-\epsilon)l$. It follows from \eqref{equ:xgmxginv} that $x_{\glg}^M=x_{\glg^{-1}}^m$. Together with \eqref{equ:distance},\eqref{ineq:coccardis}, for $(h,g)$ outside of the set $\Sigma_l$, we have
\begin{equation*}
\begin{split}
&d((hg\glg)^{-1}x',\glg^{-1}x')=d(\glg^{-1}(hg)^{-1}x',\glg^{-1}x')\\
&\leq \exp(-2\kappa(\glg^{-1})-\log d(x_{\glg^{-1}}^m,x')-\log d(x_{\glg^{-1}}^m,(hg)^{-1}x'))\leq \exp(-2(\sigma_\mu-\epsilon)l+2\epsilon l).
\end{split}
\end{equation*}
Therefore 
\begin{equation}\label{ineq:fgamma}
	|f(\glg^{-1}x',x,v,u)-f((hg\glg)^{-1}x',x,v,u)|\leq |f|_\ck d(\glg^{-1}x',(hg\glg)^{-1}x') =|f|_\ck\oexp(l). 
\end{equation}
In the bad part $\Sigma_l$, we use inequality \eqref{ineq:resindinvj} to control. Outside of $\Sigma_l$, we apply inequality \eqref{ineq:fgamma}. Thus we have
\begin{equation*}
\begin{split}
|\sum_{n\geq 0}&\int_{\sigma(hg,\glg x)>t-\kappa(\glg)\geq\sigma(g,\glg x)}
f(\glg^{-1}x',hg\glg x,\sigma(h,g\glg x),\sigma(g,\glg x)+\kappa(\glg)-t)\\
&-f((hg\glg)^{-1}x',hg\glg x,\sigma(h,g\glg x),\sigma(g,\glg x)+\kappa(\glg)-t)\dd\mu(h)\dd\mu^{*n}(g)|\\
&\leq |f|_\ck (\oexp(l)+\oexp(l)\Cut\bm 1(\glg x,t-\kappa(\glg))).
\end{split}
\end{equation*}
Then by Lemma \ref{lem:resfin}, the proof is complete.
\end{proof}
\begin{rem}[Minus case]
	Let \[\Car^- f(x',x,t)=\sum_{n\geq 0}\int_{\kappa(g)\geq t> \kappa(hg) } f((hg)^{-1}x',hgx,\kappa(hg)-\kappa(g),\kappa(g)-t)\dd\mu(h)\dd\mu^{*n}(g). \]
	Then by the same proof, we have
	\begin{repprop}{prop:rescar} Under the assumptions of Proposition \ref{prop:rescar}, we have
		\begin{equation*}
		\begin{split}
		\Car^- f(x',x,t)=\int_{X^2}\int_G\int^{-\sigma(h,y)}_{0} f(y',hy,\sigma(h,y),u)\dd u\dd\mu(h)\dd\nu(y)\dd\check\nu(y')+O_K(\delta+O_\delta/t)| f|_\lf.
		\end{split}
		\end{equation*}
	\end{repprop}
\end{rem}

\section{Main Approximation}
\label{secapproxi}

In this section, we want to complete the proof in Section \ref{secdecfou}. It remains to prove Proposition \ref{prop:mainapp} and the following Lemma \ref{lem:ntfin} and Corollary \ref{cor:ntsma}.

Recall the definitions in Section \ref{secdecfou}: Let $\mu$ be a Borel probability measure on $\slr$ with a finite exponential moment, and assume that the subgroup $\Gamma_{\mu}$ is Zariski dense. Let $\Sigma=\bigcup_{n\in\bb N}G^{\times n}$ be the symbol space of all finite sequences with elements in $G$.
Let $\bar{\mu}$ be the measure on $\Sigma$ defined by \[\bar{\mu}=\sum_{n=0}^{+\infty}\mu^{\otimes n}, \text{ where }\mu^{\otimes 0}=\delta_{\emptyset}.\]

Let the integer $\omega(g)$ be the length of an element $g$ in $\Sigma$. Let $T$ be the shift map on $\Sigma$, defined by $Tg=T(g_1,g_2,\dots,g_{\omega})=(g_1,g_2,\dots,g_{\omega-1})$, when $\omega(g)\geq 2$,  and $Tg=\emptyset$, when $\omega(g)=1,0$.
Let $L$ be the left shift map on $\Sigma$, defined by $Lg=L(g_1,g_2,\dots,g_{\omega})=(g_2,\dots,g_{\omega-1},g_{\omega})$, when $\omega(g)\geq 2$, and $Lg=\emptyset$, when $\omega(g)=1,0$.

The sets $M_t^+, N_t^+$ are defined by 
\begin{align*}
& M_t^+=\{g\in\Sigma|\ \kappa(Tg)< t\leq \kappa(g) \},\\
& N_t^+=\iota(M_t^+)=\{g\in\Sigma|\kappa(Lg)< t\leq \kappa(g) \},
\end{align*}
where $\iota(M)$ equals $\{g^{-1}|g\in M \}$ for any subset $M$ of $\Sigma$. 

Let $\check{\mu}$ be the pushforward of $\mu$ by the inverse action. It also satisfies the assumptions of Theorem \ref{thm:fourier}. By definition $\bar{\mu}(M_t^+)=\bar{\check{\mu}}(N_t^+)$.

For $x,y$ in $X$, write $s_1=\epsilon_3s$ and
\begin{equation*}
\begin{split}
M_t^+(x,y)=\{g\in M_t^+||\kappa(g)-\kappa(Tg)|< s_1,d(x_g^m,g^{-1} x)<e^{-t},
d(g^{-1} x,x),d(g^{-1} x,y)>2e^{- s_1} \}.
\end{split}
\end{equation*}

We need some regularity properties of $N_t^+$. These lemmas are of the same type as the ones with the cocycle, using the Cartan projection instead. The correspondences are: Lemma \ref{lem:ntlar} with Lemma \ref{lem:reslar}, Lemma \ref{lem:ntfin} with Lemma \ref{lem:resfin}, Lemma \ref{lem:ntind} with Lemma \ref{lem:resind}. In fact, for all the regularity properties, there are similar versions for the Cartan projection. The subadditivity is sufficient. We follow the same procedure as in  the proof for the cocycle.
\begin{lem}\label{lem:ntlar} For $s$ in $\bb R$, we have
	\begin{align}
	\bar{\mu}\{g\in N_t^+||\kappa(g)-\kappa(Lg)|>s \}=\oexp(s).
	\end{align}
\end{lem}
\begin{proof}
	Subadditivity of Cartan projection implies $\kappa(g_\omega)\geq|\kappa(g_\omega\cdots g_1)-\kappa(g_{\omega-1}\cdots g_1)|=|\kappa(g)-\kappa(Lg)|>s$ and $\kappa(Lg)\geq\kappa(g)-\kappa(g_\omega)$. Then
	\begin{align*}
		&\quad\bar{\mu}\{g\in N_t^+||\kappa(g)-\kappa(Lg)|>s \}\\
		&=\sum_{n\geq 0}\mu\otimes\mu^{* n}\{(h,g)\in G\times G|\kappa(g)<t\leq\kappa(hg),|\kappa(hg)-\kappa(g)|>s \}\\
		&\leq \sum_{n\geq 0}\mu\otimes\mu^{* n}\{(h,g)\in G\times G|t-\kappa(h)\leq\kappa(g)<t,\kappa(h)>s \}\\
		&=\int_{\kappa(h)>s}R_p(\bm 1_{[-\kappa(h),0]})(t)\dd\mu(h).
	\end{align*}
	By Lemma \ref{lem:carint} and finite exponential moment, we have
	\begin{align*}
		\bar{\mu}\{g\in N_t^+||\kappa(g)-\kappa(Lg)|>s \}\lesssim\int_{\kappa(h)>s}\max\{1,\kappa(h)^2 \}\dd\mu(h)=\oexp(s).
	\end{align*}
	The proof is complete.
\end{proof}
A special case is when $s=0$. Applying the above lemma with $\bar{\check{\mu}}$, we have
\begin{lem}\label{lem:ntfin}
	The measure $\bar{\mu}(M_t^+)=\bar{\check{\mu}}(N_t^+)$ is  uniformly bounded with $t$.
\end{lem}
The following lemma quantifies the independence of the scalar part and the angle part of residue process for the Cartan projection.
\begin{lem}\label{lem:ntind}
	For $s>0$, $t>10s$  and $x,x_o\in X$, we have
	\begin{align}
	\bar{\mu}\{g\in N_t^+|d(x_g^M,gx)\geq e^{-t} \}=\oexp(t),\\
	\bar{\mu}\{g\in N_t^+|d(gx_o,x)\leq e^{-s} \}=\oexp(s).
	\end{align}
\end{lem}
The proof of the second inequality follows the same procedure as in the proof of Lemma \ref{lem:resind}, replacing Lemma \ref{lem:reslar} and Proposition \ref{prop:cocind} with Lemma \ref{lem:ntlar} and Lemma \ref{lem:carind}. The first inequality is standard, using Principle 1 and Principle 2. When $n\in\nte$, use Corollary \ref{cor:regcon}, and when $n$ is outside of this interval, use Corollary \ref{cor:lardev} and Corollary \ref{cor:lardev1}.

Joining Lemma \ref{lem:ntlar} and Lemma \ref{lem:ntind}, we have the following corollary
\begin{cor}
Let $s>0$, $t>10s$ and let $x,y$ be in $X$. Let
\begin{equation}
\begin{split}
N_t^+(x,y)=\{g\in N_t^+||\kappa(g)-\kappa(Lg)|< s,d(x_g^M,g x)<e^{- t},
d(g x,x),d(g x,y)>2e^{- s} \}.
\end{split}
\end{equation}
Then we have
\begin{equation}
\bar{\mu}(N_t^+)-\bar{\mu}(N_t^+(x,y))=\oexp(s).
\end{equation}
\end{cor}
\begin{cor}\label{cor:ntsma}
For $s>0$, $t>10s$ and $x,y$ in $X$, we have
 \[\bar{\mu}(M_t^+)-\bar{\mu}(M_t^+(x,y))=\oexp(s).\]
\end{cor}
\begin{proof}
By definition, we have
\[\bar{\mu}(M_t^+)-\bar{\mu}(M_t^+(x,y))=\bar{\check{\mu}}(N_t^+)-\bar{\check{\mu}}(N_t^+(x,y)).\]
Applying the above corollary with $\bar{\check{\mu}}$, we have completed the proof.
\end{proof}

We start to proof Proposition \ref{prop:mainapp}. The central tool here is Lemma \ref{lem:carcoc}, which enables us to replace the cocycle with the sum of the scalar part and the angle part.
\begin{proof}[Proof of Proposition \ref{prop:mainapp}]
We first replace the distance with the cocycle. By hypothesis, we have
\[d(x_g^m,x)\geq d(g^{-1} x,x)-d(x_g^m,g^{-1} x)\geq 2e^{-  s_1}-e^{-t}\geq e^{-s_1}. \]
Using the same argument, we have $d(x_g^m,y),d(x_g^m, x)\geq e^{-  s_1}$. Then \eqref{equ:distance} and \eqref{ineq:coccardis} imply
	\begin{align*}
	d(gx,gy)=d(x,y)\exp(-\sigma(g,x)-\sigma(g,y))\leq\frac{\exp(-2\kappa(g))}{d(x_g^m,x)d(x_g^m,y)}\leq e^{-2(t- s_1)}.
	\end{align*}
Applying the Newton-Leibniz formula \eqref{equ:newcir} to $\phi$ at $gx,gy$, we have
	\begin{align*}
	\phi( gx)-\phi( gy)=\sgn(gx,gy)\int_{ gx\smallfrown gy}\phi'(\theta)\dd\theta.
	\end{align*}
Since $\kappa(g)>t>s_1$, we have $d(x_g^m,x)\geq e^{-s_1}\geq e^{-\kappa(g)}$. Then \eqref{equ:sgngxy} implies that
\[ \phi( gx)-\phi( gy))=\sgn(x,y,x_g^m)\int_{ gx\smallfrown gy}\phi'(\theta)\dd\theta.\]
We need the arc length distance $d_a(\cdot,\cdot)$ on $\bb R/\pi\bb Z$.
Since $d(gx,gy)\leq e^{-2(t-s_1)}$, for $\theta$ in the small arc $ gx\smallfrown gy$, we have $d_a(\theta, g x)\leq e^{-2(t- s_1)}$. Therefore
\begin{align}
|\phi( gx)-\phi( gy)-\sgn(x,y,x_g^m)\phi'( g x)d_a( gx, gy)|\leq |\phi''|_\infty e^{-4(t- s_1)}.
\end{align}
By equality $\sin d_a( gx, gy)=d(gx,gy)$, we have
\[|d_a( gx, gy)-d(gx,gy)|=O(d(gx,gy)^3). \]
So we can replace the arc length distance with the sine distance.
Again by hypothesis, we have $d(x_g^m,g^{-1} x)\leq e^{-  t}<d(g^{-1} x,x),d(g^{-1} x,y)$. When changing $x_g^m$ to $g^{-1} x$, the relative place with respect to $x,y$ does not change, therefore we get
\[\sgn(x,y,x_g^m)=\sgn(x,y,g^{-1} x). \]
Inequality \eqref{equ:distance}, together with the above two inequalities, implies
\begin{equation}\label{ineq:gxgy}
|\phi( gx)-\phi( gy)-\sgn(x,y,g^{-1} x)\phi'( g x)d(x,y)\exp(-\sigma(g,x)-\sigma(g,y))|\leq |\phi''|_\infty 2e^{-4(t-s_1)}.
\end{equation}

We may now replace the cocycle with the Cartan projection and the angle part. Since $$\frac{e^{-2\kappa(g)}+d(x_g^m,g^{-1} x)}{d(g^{-1} x,x)}\leq 2e^{-t+ s_1}<1/2,$$ Lemma \ref{lem:carcoc} implies that
\begin{align*}
&|\sigma(g,x)-\kappa(g)-\log d(g^{-1} x,x)|\leq 2\frac{e^{-2\kappa(g)}+d(x_g^m,g^{-1} x)}{d(g^{-1} x,x)}\leq 4e^{-t+s_1}, \\
&|\sigma(g,y)-\kappa(g)-\log d(g^{-1} x,y)|\leq 2\frac{e^{-2\kappa(g)}+d(x_g^m,g^{-1} x)}{d(g^{-1} x,y)}\leq 4e^{-t+s_1}.
\end{align*}
We have an inequality for $z_1,z_2$ in $\bb C$,
\[|e^{z_1}-e^{z_2}|\leq \max\{e^{\Re z_1},e^{\Re z_2}\}|z_1-z_2|. \]
Since $\sigma(g,x)\geq \kappa(g)+\log d(x_g^m,x)\geq t-  s_1$ and $\kappa(g)+\log d(g^{-1} x,x)\geq t-  s_1$, we have 
\begin{align*}
|\exp(-\sigma(g,x))-\exp(-\kappa(g))/d(g^{-1} x,x)|\leq e^{-t+  s_1}4e^{-t+s_1}.
\end{align*}
Therefore by inequality $|a_1a_2-b_1b_2|\leq |(a_1-b_1)a_2|+|(a_2-b_2)b_1|$, we have
\[|e^{-\sigma(g,x)-\sigma(g,y)}-e^{-2\kappa(g)}/(d(g^{-1} x,x)d(g^{-1} x,y))|\leq 8e^{-3(t-s_1)}. \]
Then by the hypothesis $|\xi|=e^{2t+s}$ and \eqref{ineq:gxgy}, we have
\begin{align*}
&\left|e^{i\xi(\phi( gx)-\phi( gy)}-e^{i\xi\phi'( g x)\sgn(x,y,g^{-1} x)d(x,y)\exp(-2\kappa(g))/(d(g^{-1} x,x)d(g^{-1} x,y))}\right|\\
\leq & |\xi||\phi( gx)-\phi( gy)-\phi'( g x)\sgn(x,y,g^{-1} x)d(x,y)\exp(-2\kappa(g))/(d(g^{-1} x,x)d(g^{-1} x,y))|\\
\leq &|\xi||\phi''|_\infty 2e^{-4(t-s_1)}+|\xi||\phi'd(x,y)||e^{-\sigma(g,x)-\sigma(g,y)}-e^{-2\kappa(g)}/(d(g^{-1} x,x)d(g^{-1} x,y))|\\
\leq &|\xi|(|\phi''|_\infty 2e^{-4(t-s_1)}+8|\phi'|_\infty e^{-3(t-s_1)})\leq 8(|\phi''|_\infty+|\phi'|_\infty)e^{-t+s+3s_1}. 
\end{align*}

Finally, for $|\cartwo_0-\cartwo|$, it suffices to add the difference 
\[|r( gx)r( gy)-r( g x)^2|\leq |r|_\infty|r'|_\infty e^{-2(t-s_1)}.\]
Then 
\begin{equation*}
|\cartwo_0-\cartwo|\leq |r|_\infty|r'|_\infty e^{-2(t-s_1)}+|r|_\infty^2(|\phi''|_\infty+|\phi'|_\infty)e^{-t+s+3s_1}=\oexp(s),
\end{equation*}
where $\oexp(s)$ does not depend on $t$, but depends on $r,\phi$. The proof is complete.
\end{proof}
\begin{rem}[Minus case]
The proof works the same for $M_t^-$.
\end{rem}

\noindent Jialun LI\\
Université de Bordeaux\\
jialun.li@math.u-bordeaux.fr

\end{document}